\documentclass[11pt]{amsart}
\usepackage{amssymb}
\usepackage{mathrsfs}
\usepackage{hyperref}
\usepackage{color}
\usepackage{graphicx}
\usepackage{epstopdf}
\usepackage{fancybox}
\usepackage{pstricks}
\usepackage{url}
\usepackage{mathtools}
\usepackage{tensor}

\date{\today}

\newcommand{\Z}{{\mathbb Z}}
\newcommand{\K}{{\mathbb K}}
\newcommand{\R}{{\mathbb R}}

\newcommand{\C}{{\mathbb C}}
\newcommand{\N}{{\mathbb N}}

\newcommand{\nn}{\nonumber}
\newcommand{\be}{\begin{equation}}
\newcommand{\ee}{\end{equation}}

\newcommand{\ol}{\overline}
\newcommand{\ti}{\tilde}

\DeclareMathOperator{\dist}{dist}

\DeclareMathOperator{\Orb}{Orb}

\newcommand{\eps}{\varepsilon}

\def\lb{\label}
\def\qq{\qquad}
\def\q{\quad}
\def\disp{\displaystyle}
\def\Proof{\noindent{\bf Proof} \quad}
\def\qed{\hfill $\Box$ \smallskip}
\def\rd{{\rm d}}

\def\x#1{(\ref{#1})}
\def\xx#1{{\rm (\ref{#1})}}

\def\ti{\tilde}
\def\qqf{\qquad \forall \ }
\def\oo{\infty}
\def\Ga{\Gamma}

\def\hh{ }
\def\dmu{\,{\rm d}\mu}
\def\Ga{\Gamma}
\def\e{\varepsilon}

\def\bea{\begin{eqnarray}}
\def\eea{\end{eqnarray}}
\def\beaa{\begin{eqnarray*}}
\def\eeaa{\end{eqnarray*}}
\def\hh{\!\!\!\!}
\def\EM{\hh &   &\hh}
\def\EQ{\hh & = & \hh}

\def\LE{\hh & \le & \hh}

\def\AND#1{\hh & #1 & \hh}

\newtheorem{theorem}{Theorem} [section]
\newtheorem{remark}[theorem]{Remark}
\newtheorem{lemma}[theorem]{Lemma}

\newtheorem{definition}[theorem]{Definition}

\def\ce{ }

\def\wds{{\widetilde \dist}}

\numberwithin{equation}{section}

\begin{document}

\title[rotation number with jump discontinuities and $\delta$-interactions]{The rotation number for almost periodic potentials with jump discontinuities and $\delta$-interactions}

\author[D. Damanik]{David Damanik}

\address{Department of Mathematics, Rice University, Houston, TX~77005, USA}
\email{\href{mailto:damanik@rice.edu}{damanik@rice.edu}}

\author[M. Zhang]{Meirong Zhang}

\address{Department of Mathematical Sciences, Tsinghua University,
Beijing 100084, China}
\email{\href{mailto:zhangmr@tsinghua.edu.cn}{zhangmr@tsinghua.edu.cn}}

\author[Z. Zhou]{Zhe Zhou$^*$}

\address{Academy of Mathematics and Systems Science,  Chinese Academy of Sciences,  Beijing 100190, China}
\address{School of Mathematical Sciences, University of Chinese Academy of Sciences, Beijing 100049, China}
\email{\href{mailto:zzhou@amss.ac.cn}{zzhou@amss.ac.cn}}

\thanks{$^*$ Corresponding author}
\thanks{D.\ D.\ was supported by Simons Fellowship $\# 669836$ and NSF grants DMS--1700131 and DMS--2054752}
\thanks{M.\ Z.\ was supported by the National Natural Science Foundation of China (Grant No. 11790273)}
\thanks{Z.\ Z.\ was supported by the National Natural Science Foundation of China (Grant Nos. 12271509, 12090010, 12090014)}

\date{\today}



\begin{abstract}
We consider one-dimensional Schr\"odinger operators with generalized almost periodic potentials with jump discontinuities and $\delta$-interactions. For operators of this kind we introduce a rotation number in the spirit of Johnson and Moser. To do this, we introduce the concept of almost periodicity at a rather general level, and then the almost periodic function with jump discontinuities and $\delta$-interactions as an application.
\end{abstract}

\maketitle

\tableofcontents

\section{Introduction}

The study of the spectral properties of Schr\"odinger operators with almost periodic potentials has been an active research area for roughly half a century. Some of the exciting features that have been exhibited are nowhere dense (i.e., Cantor) spectra and the possibility for any of the standard spectral types (i.e., pure point, singular continuous, and absolutely continuous) to occur within this class of models. This shows that the spectral phenomena are richer than in the classical subclass of periodic potentials, which have been studied much longer.

Our understanding of these issues is much more complete in one space dimension, although some exciting progress has been made in higher dimensions as well; the reader may start exploring the existing theory by consulting, for example, \cite{CL90, D17, KPS20+, PF92} and references therein. The key difference between the one-dimensional case and the higher-dimensional case is the fact that the former admits a generalized eigenvalue equation that is a linear second-order ordinary differential equation, and hence has a two-dimensional solution space for any given energy. This should be contrasted with the fact that in space dimensions at least two, the generalized eigenvalue equation is a partial differential equation and the solution space is infinite-dimensional. This distinction is important because the spectral questions one is interested in can be related to the behavior of the solutions of the generalized eigenvalue equation.

Thus, in the case of one space dimension, it is a worthwhile goal to understand the behavior of these solutions $\psi$, and the two-dimensionality then leads one to study their dependence on the space variable, $x \in \R$, in the plane, $(\psi'(x), \psi(x))^T \in \R^2$. Choosing polar-type coordinates, which are usually referred to as \emph{Pr\"ufer variables}, one can study the growth and the rotation of the vector $(\psi'(x),\psi(x))$ around the origin of $\R^2$ as $x$ grows. The almost periodicity of the potential is well known to yield a uniquely ergodic dynamical system, namely the \emph{hull}, which is the uniform closure of the set of translates and which turns out to be a compact abelian group, together with the $\R$ translation action and the normalized Haar measure. As a consequence, the average amount of rotation per unit step can be defined as the limit of Birkhoff-type averages, which exists uniformly on the hull due to unique ergodicity. The resulting limit is called the \emph{rotation number}, $\rho(E)$, at the energy $E$ in question. It turns out that $\rho(\cdot)$ is constant in a suitable neighborhood of $E$ if and only if $E$ belongs to the complement of the spectrum. In particular, each gap of the spectrum can then be labeled by the constant value $\rho$ takes on it. Additionally, the possible labels that can in principle occur are completely determined by the hull, either via $K$-theory or the Schwartzman asymptotic cycle. This realization is crucial in the study of the topological structure of the spectrum, and in particular when proving that the spectrum is generically nowhere dense. We refer the reader to the landmark papers by Johnson-Moser \cite{JM82} and Johnson \cite{J86} for the definition and study of the rotation number for one-dimensional almost periodic Schr\"odinger operators and its application to gap labeling.

This paper is motivated by the desire to generalize almost periodic potentials by adding suitable almost periodic local point interactions. The discussion above then suggests that an important first step in the analysis of the resulting operators is the definition and study of the rotation number $\rho(E)$ for $E \in \R$. This is precisely what we carry out in the present paper. The application of the rotation number we define here to the spectral analysis of these generalized almost periodic Schr\"odinger operators in one space dimension will be presented in a forthcoming paper. 

Let us describe the models we will be interested in. Following and extending the papers \cite{JM82, ZZ11, DZ}, we consider the Schr\"{o}dinger operator $H_{q,V,\Gamma}$ in $L^2(\R)$ given by
    \be \lb{KP}
    H_{q,V,\Gamma}\psi(x):=-\psi''(x)+\left(q(x)+\sum_{i \in \Z} v_i \delta(x - x_i)\right) \psi(x), \qq x \in \R,
    \ee
where $q(x) \in PC_u(\R)$ is a piecewise continuous almost periodic function, $V=\{v_i\}_{i \in \Z} \in \ell^{\oo}(\Z)$ is an almost periodic bi-sequence, and $\Gamma=\{x_i\}_{i \in \Z} \in L(\Z)$ is an almost periodic point set, while $\delta(x-x_i)$ denotes the Dirac $\delta$-function at $x_i$. In such a setting, $q(x)$ and $\sum_{i \in \Z} v_i \delta(x - x_i)$ can be regarded as the absolutely continuous part and the pure point part of the potential in the measure sense, respectively. Let $E \in \R$. The eigenvalue equation
$$
H_{q,V,\Gamma}\psi= E \psi
$$
can be written as
    \begin{equation} \label{sys}
    \left\lbrace \begin{array}{ll}
    \disp\frac{\rd}{\rd x}\left(
                       \begin{array}{c}
                         \psi' \\
                         \psi \\
                       \end{array}
                     \right) =\left(
                                \begin{array}{cc}
                                  0 & q(x)- E \\
                                  1 & 0 \\
                                \end{array}
                              \right)\left(
                                       \begin{array}{c}
                                         \psi' \\
                                         \psi \\
                                       \end{array}
                                     \right), \qquad &x \in \mathbb{R}\setminus \Gamma,\\
    \left(
      \begin{array}{c}
        \psi'(x_i+) \\
        \psi(x_i+) \\
      \end{array}
    \right)= \left(
               \begin{array}{cc}
                 1 & v_i \\
                 0 & 1 \\
               \end{array}
             \right)\left(
                      \begin{array}{c}
                         \psi'(x_i-) \\
                        \psi(x_i-) \\
                      \end{array}
                    \right), \qquad &x_i \in \Gamma.
    \end{array} \right.
    \end{equation}

The system \x{sys} can be regarded as an impulsive differential equation. There is a large number of works on systems with impulses in which the behavior of solutions is studied, such as periodicity, almost-periodicity, stability and so on; see the monograph by Samoilenko and Perestyuk with a supplement by Trofimchuk \cite{SPT95} and references therein. Different from those works, we will focus on the long time behavior of solutions of \x{sys}, and introduce the rotation number in the spirit of Johnson and Moser \cite{JM82} for \x{KP}.

The paper is organized as follows. We begin in Section~\ref{sec.2} with general considerations centered around the concept of almost periodicity. The presentation is at a rather general level, but as a primary application we have the construction and discussion of the hull associated with our generalized almost periodic potential in mind. The latter application appears in Section~\ref{se-apf}. The next step is to discuss the solutions of \x{sys} from a Pr\"ufer variable perspective that is amenable to unique ergodicity considerations; this is carried out in Section~\ref{sec.4}. Finally, our discussion culminates in Section~\ref{sec.5} in the definition and discussion of the rotation number for the models we consider in this paper.

\bigskip
Throughout this paper, we adopt the following notations:
\begin{itemize}
\smallskip
\item $\N_0:=\N \cup \{0\}$; $\R^+_0:=\R^+\cup\{0\}$;
\smallskip
\item $\mathrm{e}$ denotes the Euler number; $\mathrm{i}$ denotes the imaginary unit that is different from the index $i$;
\smallskip
\item $\K$ denotes either $\R$ or $\C$, depending on the setting. All functions and bi-sequences are $\K$-valued unless stated otherwise;
\smallskip
\item $L(\Z)$: the set of all discrete point sets in the real axis;
\smallskip
\item $\ell^{\oo}(\Z)$: the space of all bounded bi-sequences;
\smallskip
\item $\ell^2(\Z)$: the space of all square integrable bi-sequences;
\smallskip
\item $C(\R)$: the space of all continuous functions;
\smallskip
\item $C_b(\R)$: the subspace of $C(\R)$ consisting of all bounded functions;
\smallskip
\item $C_u(\R)$: the subspace of $C_b(\R)$ consisting of all uniformly continuous functions;
\smallskip
\item $PC(\R)$: the space of all piecewise continuous functions with jump discontinuities at a discrete point set $\Gamma \in L(\Z)$;
\smallskip
\item $PC_b(\R)$: the subspace of $PC(\R)$ consisting of all bounded functions;
\smallskip
\item $PC_u(\R)$: the subspace of $PC_b(\R)$ consisting of all functions that are uniformly continuous on $\R \setminus \Gamma$, i.e., 
    for any $\e >0$, there exists $\delta_\e>0$ such that $|f(x_1)-f(x_2)|<\e$ when $|x_1-x_2|<\delta_\e$ and $x_1, \ x_2$ belong to the same interval from $\R \setminus \Gamma$;
\end{itemize}

\medskip

%
%
%
%
%

\section{Almost Periodicity}\label{sec.2}

In \cite{DZ}, we used a unified approach to introduce almost periodicity in which the isometry of shift actions is crucial. Here we will improve this approach by using a weaker condition than isometry and establish further properties of almost periodicity. Some related ideas can be founded in \cite{Zh}.

%
%
%
%
%

\subsection{Almost Periodic Point}

Let $(Y,\dist)$ be a complete metric space. We consider a $\Z$ action on $Y$ by shifts and denote for $y \in Y$ and $\tau \in \Z$ the corresponding shifted element in $Y$ by $y \cdot \tau$. This shift action satisfies the following conditions:
\begin{itemize}
\smallskip
\item group structure:
    \be \lb{gp-y}
    y \cdot 0=y, \mbox{~and~} y \cdot (\tau_1+\tau_2)= (y \cdot \tau_1) \cdot \tau_2, \qqf y \in Y, \ \tau_1, \ \tau_2 \in \Z;
    \ee
\item equicontinuity:
    \bea
    \EM\mbox{for~any~}\e>0, \mbox{~there~exists~}\delta_\e>0 \mbox{~such~that}\nn \\
    \EM \lb{we-is} \mbox{if~} \dist(y_1,y_2)<\delta_\e, \mbox{~then~} \dist(y_1 \cdot \tau ,y_2 \cdot \tau) < \e \mbox{~for~all~} \tau \in \Z.
    \eea
\end{itemize}

For $y \in Y$, denote the \emph{orbit} of $y$ by
    \[
    \Orb(y):=\{y \cdot \tau: \tau \in \Z\} \subset Y,
    \]
and the \emph{hull} of $y$ by
    \[
    \mathrm{H}(y):= \ol{\Orb(y)}^{(Y,\dist)}.
    \]
A set $A \subset \Z$ is said to be \emph{relatively dense} (with window size $\ell$) if there exists $\ell \in \N$ such that
    \[
    A \cap [a, a + \ell] \neq \emptyset, \qqf a \in \N.
    \]

\begin{definition} \label{un-ap}
We say that $y \in Y$ is almost periodic if one of the following conditions holds:
\begin{itemize}
\smallskip
\item[{\romannumeral1}):] for any $\eps > 0$,
		\(
		\mathrm{P}(y,\eps) := \left\{\tau \in \Z :\dist( y \cdot \tau ,y) < \eps \right\}
		\)
		is relatively dense in $\Z$;
\smallskip		
\item[{\romannumeral2}):] the hull of $y$ is compact;
\smallskip		
\item[{\romannumeral3}):] for any sequence $\{\ti \tau_k \}_{k \in \N}\subset \Z$, one can extract a subsequence $\{ \tau_k \} \subset \{ \ti \tau_k \}$ such that $\{ y \cdot \tau_k\}$ is convergent in $(Y, \dist)$, i.e., $\Orb(y)$ is relatively compact.
\end{itemize}
\end{definition}

The concept of almost periodic points has been introduced for a topological action of a topological semigroup on a locally compact topological space; see \cite{Go47, CD16}. The difference is that here we do not assume that $Y$ is locally compact, but require that the shift action satisfies the equicontinuity. Under conditions \x{gp-y} and \x{we-is}, we can show the equivalence of these definitions.

\begin{theorem} \lb{th-eq}
The conditions {\rm {\romannumeral1}), {\romannumeral2})} and {\rm {\romannumeral3})} in {\rm Definition \ref{un-ap}} are equivalent.
\end{theorem}

\begin{proof} The argument to show the equivalence between {\rm {\romannumeral2})} and {\rm {\romannumeral3})} is the same as in the proof of \cite[Theorem 2.4]{DZ}. We only consider the equivalence between {\rm {\romannumeral1})} and {\rm {\romannumeral2})}. 

\medskip
\fbox{{\rm {\romannumeral1})} $\Longrightarrow$ {\rm {\romannumeral2})}}\ : By \cite[Therem 3.16.1]{Di69} it suffices to show that $(\mathrm{H}(y), \dist)$ is complete and totally bounded. Since $\mathrm{H}(y)$ is closed in the complete space $(Y,\dist)$, $(\mathrm{H}(y), \dist)$ is complete as well. Hence we need only to prove that $(\mathrm{H}(y), \dist)$ is totally bounded. It suffices to show that $\Orb(y)$ is totally bounded. For any $\e > 0$, there exists $\delta_\e>0$ such that condition \x{we-is} holds. By Definition \ref{un-ap} {\romannumeral1}), for such $\delta_\e$, $\mathrm{P}(y,\delta_\e)$ is relatively dense in $\Z$. Thus there exists $\ell_{\e} \in \N$ such that for any $a \in \Z$, we have
    \[
    \mathrm{P}(y,\delta_\e) \cap [-a,-a+\ell_{\e}] \neq \emptyset.
    \]
Let $-a + b_{a,\e} \in \mathrm{P}(y,\delta_\e) \cap [-a, -a + \ell_{\e}]$, where $b_{a,\e} \in [0, \ell_{\e}] \cap \Z$ depends on the parameters $a$ and $\e$. Since
    \[
    \dist( y \cdot (-a+b_{a,\e}),y)<\delta_\e,
    \]
it follows from \x{gp-y} and \x{we-is} that
    \be \lb{th-eq1}
    \dist( y \cdot b_{a,\e}, y \cdot a)< \e, \qqf a \in \Z.
    \ee
We construct a finite set $A_{\e} \subset \Orb(y)$ by
    \[
    A_{\e} := \{ y \cdot i : i = 0, 1, \cdots ,\ell_{\e}\}.
    \]
By \x{th-eq1}, we obtain that $\Orb(y)$ is totally bounded.

\medskip

\fbox{{\rm {\romannumeral2})} $\Longrightarrow$ {\rm {\romannumeral1})}}\ : Since $\mathrm{\mathrm{H}}(y)$ is compact, by \cite[Therem 3.16.1]{Di69} it follows that $\Orb(y)$ is totally bounded. Again by using $\delta_\e$ in \x{we-is}, we know that there exists a finite subset, denoted by $I_{\e} := \{ \tau_{i} \in \Z : i = 1, 2, \cdots, n_{\e} \}$ such that
    \be \lb{th-eq3}
    \dist (y \cdot a, y \cdot \tau_{i_{a}}) < \delta_\e, \qqf a \in \Z,
    \ee
where $\tau_{i_{a}} \in I_{\e}$ depends on the parameter $a$. It follows from \x{gp-y}, \x{we-is} and \x{th-eq3} that 
    \be \lb{th-eq5}
    \dist(y, y \cdot (-a+\tau_{i_{a}})) < \e, \qqf a \in \Z.
    \ee
Let us denote $L_{\e} := \max_{ 1 \le i \le n_{\e} } | \tau_{i} |$. Then
    \[
    -a - L_{\e} \le -a + \tau_{i_a} \le -a + L_{\e}, \qqf a \in \Z.
    \]
Combining this with \x{th-eq5}, we have
    \[
    \mathrm{P}(y,\eps) \cap [ -a - L_{\e}, -a + L_{\e} ] \neq \emptyset.
    \]
Thus $\mathrm{P}(y,\eps)$ is relatively dense with the choice of $\ell_{\e} = 2 L_{\e}$.
\end{proof}

\begin{remark} \lb{re-Go}
{\rm The equivalence between conditions {\rm {\romannumeral1}) and {\romannumeral3})} in {\rm Definition \ref{un-ap}} is similar to a theorem of Gottschalk. Due to \cite[Theorem 1]{Go46}, that result states that when $Y$ is a uniform space, the total boundedness of $\Orb(y)$ is equivalent to the almost periodicity of $y$ under the equicontinuity condition. Since we have assumed that $Y$ is a metric space, the equivalence between total boundedness and relative compactness holds automatically, and we may characterize the convergence by sequences instead of some uniformity. }
\end{remark}

We say that $y \in Y$ is an \emph{equicontinuous point} if the shift action restricted on $\Orb(y)$ satisfies condition \xx{we-is}. Denote the subset of $Y$ consisting of all equicontinuous points by $Y_{ec}$.

\begin{remark} \lb{re-ec}
{\rm  Recalling the proof of Theorem \ref{th-eq}, we only use the equicontinuity of the shift action on $\Orb(y)$. Once we do not have the equicontinuity condition \x{we-is} on the whole space $Y$, we may define the almost periodicity for $y \in Y_{ec}$ instead of $y \in Y$.
}
\end{remark}

We say that $y \in Y$ is a \emph{complete point} if $(\mathrm{H}(y),\dist)$ is a complete subspace. Denote the subset of $Y$ consisting of all complete points by $Y_{co}$.

\begin{remark} \lb{re-co}
{\rm  Again recalling the proof of Theorem \ref{th-eq}, we only use the completeness of $H(y)$. This means that even if the whole space $Y$ is not complete, we may still define the almost periodicity for $y \in Y_{co}$ instead of $y \in Y$.
}
\end{remark}

Denote the subset of $Y$ consisting of all almost periodic points by $Y_{ap}$. 

\begin{lemma} \lb{ap-com}
$(Y_{ap},\dist)$ is a complete metric space.
\end{lemma}

\begin{proof}
Let $\{y_i\}_{i \in \N} \subset Y_{ap}$ be a Cauchy sequence. Since $Y$ is complete, there exists $y_0 \in Y$ such that
    \be \lb{ap-com1}
    \lim_{i \to +\infty} \dist(y_i,y_0)=0.
    \ee
We assert that $y_0 \in Y_{ap}$. Indeed, since $y_1 \in Y_{ap}$, we know by Definition \ref{un-ap} {\romannumeral3}) that for any sequence $\{\ti \tau_k \}_{k \in \N}\subset \Z$, there exists a subsequence $\{ \tau^1_k \} \subset \{ \ti \tau_k \}$ such that $\{ y_1 \cdot \tau^1_k\}_{k \in \N}$ is convergent in $(Y, \dist)$. For $y_{i+1} \in Y_{ap}$ and $\{\tau^i_k\}$, $i \in \N$, repeating the process we have a subsequence $\{ \tau^{i+1}_k \} \subset \{ \tau^{i}_k \}$ such that $\{ y_{i+1} \cdot \tau^{i+1}_k\}_{k \in \N}$ is convergent in $(Y, \dist)$. It follows from the diagonalization process that there exists a subsequence $\{ \tau^k_k \} \subset \{ \ti \tau_k \}$ such that
    \be \lb{ap-com2}
    \Bigl\{y_{i} \cdot \tau^{k}_k\Bigr\}_{k \in \N} \mbox{~is~convergent~in~}(Y, \dist)\mbox{~for~all~}i \in \N.
    \ee
Replacing $\e$ by $\e/3$ in \x{we-is}, we take a number $\delta_{\e/3}$. By \x{ap-com1} there  exists $i_\e \in \N$ such that
    \[
    \dist(y_i,y_0)< \delta_{\e/3}, \q \mbox{for~}i \ge i_\e.
    \]
This implies that
    \be \lb{ap-com3}
    \dist(y_{i} \cdot \tau ,y_0 \cdot \tau) < \e/3,\q \mbox{for~} i \ge i_\e \mbox{~and~}\tau \in \Z.
    \ee
Let $i=i_\e$ in \x{ap-com2}. We know that for any $\e>0$, there exists $k_\e \in \N$ such that
    \be \lb{ap-com4}
    \dist\Bigl(y_{i_\e} \cdot \tau^{k_1}_{k_1},y_{i_\e} \cdot \tau^{k_2}_{k_2}\Bigr) <\e/3, \q \mbox{for~}k_1,\ k_2 \ge k_\e.
    \ee
Then by \x{ap-com3} and \x{ap-com4} we have
    \beaa
    \EM \dist\Bigl(y_0\cdot \tau^{k_1}_{k_1}, y_0\cdot \tau^{k_2}_{k_2}\Bigr) \\
    \LE \dist\Bigl(y_{i_\e}\cdot \tau^{k_1}_{k_1}, y_0\cdot \tau^{k_1}_{k_1}\Bigr)+\dist\Bigl(y_{i_\e}\cdot \tau^{k_2}_{k_2}, y_0\cdot \tau^{k_2}_{k_2}\Bigr)+\dist\Bigl(y_{i_\e} \cdot \tau^{k_1}_{k_1},y_{i_\e} \cdot \tau^{k_2}_{k_2}\Bigr) \\
    \AND < \e/3+\e/3+\e/3 = \e.
    \eeaa
Thus $\Bigl\{y_0\cdot \tau^{k}_{k}\Bigr\}_{k \in \N}$ is convergent in $(Y, \dist)$. By Definition \ref{un-ap} {\romannumeral3}) we have the assertion.
\end{proof}

\begin{remark} \lb{re-com}
{\rm  Although we may define the almost periodicity for $y \in Y_{ec} \cap Y_{co}$, we do not know whether Lemma \ref{ap-com} is still correct in the absence of condition \x{we-is} and the completeness of $Y$.
}
\end{remark}

%
%
%
%
%

\subsection{Compact Abelian Topological Group}

We focus on the hull of $y \in Y_{ap}$ and equip $\mathrm{H}(y)$ with a group operation as follows. Let
    \be \label{yk}
    y_k:=\lim_{i\to +\oo} y\cdot \tau^k_i  \in \mathrm{H}(y), \qq k=1, \, 2.
    \ee
Then we define the group operation by
    \be \lb{pr-y}
    y_1 \times y_2 := \lim_{i\to +\oo} y\cdot (\tau^1_i+\tau^2_i).
    \ee
The element $y_1 \times y_2$ is well defined. Indeed, for $\delta_{\e/2}$ in \x{we-is}, by \x{yk} there exists a common $i_\e \in \N$ such that
    \be \lb{y12}
    \dist(y \cdot \tau^1_{i_1}, y \cdot \tau^1_{i_2}) < \delta_{\e/2} \mbox{~and~} \dist(y \cdot \tau^2_{i_1}, y \cdot \tau^2_{i_2}) < \delta_{\e/2}, \q \mbox{for~} i_1,\ i_2 \ge i_\e.
    \ee
By \x{gp-y}, \x{we-is} and \x{y12}, we know that when $i_1,\ i_2 \ge i_\e$, we have
    \beaa
    \EM \dist\Bigl(y \cdot (\tau^1_{i_1}+\tau^2_{i_1}), y \cdot (\tau^1_{i_2}+\tau^2_{i_2})\Bigr) \\
    \LE \dist\Bigl((y \cdot \tau^2_{i_1})\cdot \tau^1_{i_1}, (y \cdot \tau^2_{i_2}) \cdot \tau^1_{i_1}\Bigr)+\dist\Bigl((y \cdot \tau^1_{i_1}) \cdot \tau^2_{i_2}, (y \cdot \tau^1_{i_2}) \cdot \tau^2_{i_2}\Bigr) \\
    \AND < \e/2+\e/2 = \e.
    \eeaa
This implies that $\Bigl\{y \cdot (\tau^1_i+\tau^2_i)\Bigr\}_{i \in \N}$ is a Cauchy sequence in the complete space $(Y,\dist)$. Then the limit \x{pr-y} exists. Moreover, $y_1 \times y_2$ is independent of the choice of the sequences $\{\tau^k_i\}_{i \in \N}$ in \x{yk}. In fact, suppose that there exist other sequences $\{\ti \tau^k_i\}_{i \in \N}$ such that
    \be \lb{tiyk}
    y_k=\lim_{i\to +\oo} y\cdot \ti \tau^k_i, \qq k=1,2.
    \ee
Again for $\delta_{\e/2}$ in \x{we-is}, by \x{yk} and \x{tiyk} there exists a common $i_\e \in \N$ such that
    \be \lb{tiy12}
    \dist(y \cdot \tau^k_{i}, y \cdot \ti \tau^k_{i}) < \delta_{\e/2}, \q \mbox{for~} i\ge i_\e,\  k=1,2.
    \ee
By \x{gp-y}, \x{we-is} and \x{tiy12}, we know that when $i \ge i_\e$, one has
    \beaa
    \EM \dist\Bigl(y \cdot (\tau^1_{i}+\tau^2_{i}), y \cdot (\ti \tau^1_{i}+\ti \tau^2_{i})\Bigr) \\
    \LE \dist\Bigl((y \cdot \tau^1_{i})\cdot \tau^2_{i}, (y \cdot \ti \tau^1_{i}) \cdot \tau^2_{i}\Bigr)+\dist\Bigl((y \cdot \tau^2_{i}) \cdot \ti \tau^1_{i},(y \cdot \ti \tau^2_{i}) \cdot \ti\tau^1_{i}\Bigr) \\
    \AND < \e/2+\e/2 = \e.
    \eeaa
Since $\e$ is arbitrary, we have the desired assertion.

These considerations also suggest that the inverse of $y_1$ will be given by
    \be \lb{inv-y}
    (y_1)^{-1}:= \lim_{i\to +\oo} y\cdot (-\tau^1_i).
    \ee
Indeed, using a similar argument as above, we conclude that $(y_1)^{-1}$ is independent of the choice of the sequence $\{\tau^1_i\}_{i \in \N}$, and hence is well defined. Moreover, $(y_1)^{-1}$ is inverse to $y_1$.

Denote the time-one shift $\ti y \cdot 1$ by $T(\ti y)$, where $\ti y=\disp \lim_{i\to +\oo} y\cdot \ti \tau_i \in \mathrm{H}(y)$. Then we have the following results.

\begin{lemma} \lb{atg} For $y \in Y_{ap}$, one has
\begin{itemize}
\smallskip
\item[{\romannumeral1}):] $\mathrm{H}(\ti y)=\mathrm{H}(y)$ for each $\ti y \in \mathrm{H}(y)$;
\smallskip
\item[{\romannumeral2}):] $(\mathrm{H}(y),\times ,^{-1})$ is a compact abelian topological group;
\smallskip
\item[{\romannumeral3}):] $T: \mathrm{H}(y) \to \mathrm{H}(y)$ is uniquely ergodic with the Haar measure, denoted by $\nu_{y}$, being the only invariant measure; and
\smallskip
\item[{\romannumeral4}):] for any continuous function $f: \mathrm{H}(y) \to \K$,
    \be \lb{mv-atg}
    \lim_{n_2-n_1 \to +\infty} \frac{1}{n_2-n_1} \sum_{\tau=n_1}^{n_2-1} f(\ti y \cdot \tau)=\int_{\mathrm{H}(y)} f \rd \nu_{y},
    \ee
uniformly for all $\ti y \in \mathrm{H}(y)$.
\end{itemize}
\end{lemma}

\begin{proof}
\fbox{{\rm {\romannumeral1})}}\ : For $\delta_{\e}$ in \x{we-is}, there exists $i_\e \in \N$ such that
    \be \lb{atg0}
    \dist(y \cdot \ti\tau_i, \ti y)< \delta_\e, \q \mbox{for~} i \ge i_\e.
    \ee
For all $\tau \in \Z$, by \x{gp-y} and \x{we-is} we have
    \be \lb{atg01}
    \dist(y \cdot (\ti\tau_i+\tau), \ti y \cdot \tau)<\e, \q \mbox{for~} i \ge i_\e.
    \ee
This implies that $\Orb(\ti y) \subset \mathrm{H}(y)$. Because $\mathrm{H}(y)$ is closed, we obtain $\mathrm{H}(\ti y) \subset \mathrm{H}(y)$. Conversely, from \x{gp-y}, \x{we-is} and \x{atg0} we get
    \[
    \dist(\ti y \cdot (-\ti\tau_i), y)< \e, \q \mbox{for~} i \ge i_\e.
    \]
This implies that $y \in \mathrm{H}(\ti y)$. Using the argument above, we obtain $\mathrm{H}(y) \subset \mathrm{H}(\ti y)$. The proof of {\romannumeral1}) is completed.

\fbox{{\rm {\romannumeral2})}}\ : It is obvious by \x{pr-y} that $\mathrm{H}(y)$ is an abelian group with the identity element $y$. One needs to show that the group operations $\times$ and $^{-1}$ are continuous. We make the following claim:
    \bea
    \EM\mbox{for~any~}\e>0, \mbox{~there~exists~}\delta_\e>0 \mbox{~such~that}\nn \\
    \EM \lb{eq-hu} \mbox{if~} \dist(y_1,y_2)<\delta_\e, \mbox{~then~}\dist(y_1 \times y_3 ,y_2 \times y_3) < \e \mbox{~for~} y_k \in \mathrm{H}(y).
    \eea
Indeed, by \x{yk} and the continuity of metric, we have
    \be \lb{atg1}
    \lim_{i \to +\infty} \dist(y \cdot \tau^1_i, y \cdot \tau^2_i)= \dist(y_1,y_2).
    \ee
Replacing $\e$ by $\e/2$ in \x{we-is}, we take $\delta_{\e/2}$. When $\dist(y_1,y_2)<\delta_{\e/2}$, by \x{atg1} there exists $i_\e \in \N$ such that
    \[
    \dist(y \cdot \tau^1_i, y \cdot \tau^2_i) < \delta_{\e/2}, \q \mbox{for~} i \ge i_\e.
    \]
This implies that
    \[
    \dist\Bigl(y \cdot (\tau^1_i+\tau^3_i), y \cdot (\tau^2_i+\tau^3_i)\Bigr) < \e/2.
    \]
By \x{pr-y} and the continuity of the metric, the claim \x{eq-hu} is deduced and $\delta_{\e/2}$ is the desired number.

Let $y_k, \ti y_k \in \mathrm{H}(y), \ k=1,2$. We take $\delta_{\e/2}$ in \x{eq-hu}. When $\dist(y_1,\ti y_1)<\delta_{\e/2}$ and  $\dist(y_2, \ti y_2)<\delta_{\e/2}$, it follows from \x{eq-hu} that
    \beaa
    \EM \dist(\ti y_1 \times \ti y_2,y_1 \times y_2) \\
    \LE \dist(\ti y_1 \times \ti y_2,y_1 \times \ti y_2)+\dist(y_1 \times \ti y_2,y_1 \times y_2)\\
    \AND < \e/2+\e/2 = \e.
    \eeaa
This proves the continuity of the operation $\times$. Similarly, we can obtain the continuity of the operation $^{-1}$. The proof of {\romannumeral2}) is completed.

\fbox{{\rm {\romannumeral3})}}\ :  By \x{gp-y} and \x{atg01} we have
    \[
    T^{\tau}(\ti y)=\ti y \cdot \tau= \lim_{i \to +\infty} y \cdot(\ti \tau_i+\tau), \qqf \tau \in \Z.
    \]
By \x{pr-y} we have $T(\ti y) =(y \cdot 1)\times \ti y$. This implies that $T$ is a rotation on the compact abelian topological group $\mathrm{H}(y)$. Furthermore, due to \x{gp-y} and {\romannumeral1}), $T$ is a minimal rotation. Thus the proof of {\romannumeral3}) is completed by \cite[Theorem 6.20]{Wa82}.

\fbox{{\rm {\romannumeral4})}}\ : Applying a standard consequence of unique ergodicity, \cite[Theorem 6.19]{Wa82}, to $T: \ti y \mapsto \ti y \cdot 1$, the statement {\romannumeral4}) follows readily.
\end{proof}

\begin{definition} \label{un-mv}
For any $y \in Y_{ap}$ and any continuous function $f: \mathrm{H}(y) \to \K$, we call $\int_{\mathrm{H}(y)} f \rd \nu_{y}$ the mean value of $y$ with respect to $f$, and denote it by $\mathrm{M}_f(y)$.
\end{definition}

\begin{remark} \lb{atg-co}
{\rm  Similar to Remark \ref{re-co}, Lemma \ref{atg} may be established for $y \in Y_{co}$ in the absence of completeness of $Y$.
}
\end{remark}

We consider the case $Y:=\ell^{\oo}(\Z)$ and define a metric $\ell^{\oo}(\Z) \times \ell^{\oo}(\Z) \to \R^+_0$ by
    \be \lb{dist-v}
    \dist(V_1,V_2):=\|V_1-V_2\|_{\infty}=\sup_{i \in \Z}|v^1_i-v^2_i|,
    \ee
where $V_k:=\{v^k_i\}_{i \in \Z} \in \ell^{\oo}(\Z), \ k=1,2$. The following is well known.

\begin{lemma} \lb{bs-com}
$(\ell^{\oo}(\Z),\dist)$ is a complete metric space.
\end{lemma}

For $V = \{ v_i \}_{i \in \Z} \in \ell^{\oo}(\Z)$ and $\tau \in \Z$, the shift of $V$ is defined by
    \be \lb{sft-v}
    V \cdot \tau:= \{ v_{i + \tau}\}_{i \in \Z}.
    \ee
Obviously for $V_k \in \ell^{\oo}(\Z),\ k=1,2$, we have
    \be \lb{bs-is}
    \dist(V_1 \cdot \tau, V_2 \cdot \tau)=\dist(V_1,V_2), \qqf \tau \in \Z.
    \ee
This means that $(\ell^{\oo}(\Z),\dist)$ satisfies the isometry condition. Then Definition \ref{un-ap} defines \emph{almost periodic bi-sequences}. We denote by $\ell_{ap}(\Z)$ the space of all almost periodic bi-sequences. By Lemma \ref{ap-com}, we know that $(\ell_{ap}(\Z),\dist)$ is a complete space.

Introduce the function $f_0: \ell_{ap}(\Z) \to \K$ by
    \[
    f_0(V):=v_0, \q \mbox{for~} V=\{v_i\}_{i \in \Z} \in \ell_{ap}(\Z).
    \]
It is easy to see that $f_0$ is continuous. Then by Definition \ref{un-mv}, we have the following result.

\begin{lemma} \lb{lap-mv}
Let $V=\{v_i\}_{i \in \Z} \in \ell_{ap}(\Z)$. Then the limit
    \[
    \mathrm{M}_{f_0}(V)=\lim_{n_2-n_1 \to +\oo} \frac{1}{n_2-n_1} \sum_{\tau=n_1}^{n_2-1}f_0(V \cdot \tau) =\lim_{n_2-n_1 \to +\oo} \frac{1}{n_2-n_1} \sum_{\tau=n_1}^{n_2-1}v_\tau
    \]
exists. We call it the mean value of $V$ and denote it by $\mathrm{M}(V)$.
\end{lemma}

\medskip
%
%
%
%
%

\section{Almost Periodic Functions with Jump Discontinuities and $\delta$-Interactions} \lb{se-apf}

The concept of an almost periodic function is well known; see the monographs \cite{Be54, Bo56, LZ82}. In this section we introduce almost periodic functions with jump discontinuities and $\delta$-interactions denoted by
    \be \lb{pcde}
    f(x)+\sum_{i \in \Z} v_i \delta(x - x_i), 
    \ee
where $\Gamma=\{x_i\}_{i \in \Z} \in L(\Z)$, $f(x) \in PC_u(\R)$ is a function with jump discontinuities at points of $\Gamma$ and $V=\{v_i\}_{i \in \Z} \in \ell^{\oo}(\Z)$, while $\delta(x-x_i)$ is the Dirac $\delta$-function at $x_i$. In \cite{SPT95} the authors introduced piecewise continuous almost periodic (for short, p.c.a.p.) functions with first kind of discontinuities at the (possible) points of $\Gamma=\{x_i\}_{i \in \Z} \subset \R$ in which the family of sequences $\{x_{i+j}-x_i\}_{i \in \Z}$ is \emph{equipotentially almost periodic} for all $j\in \Z$. Under the separation condition $\inf_{i \in \Z}(x_{i+1}-x_i) > 0$ such a point set is a so-called Wexler sequence; see \cite[Definition 2.11]{QY19}. If $\R$ is regarded as a locally compact abelian group, this is a \emph{modulated lattice} as introduced in \cite[Definition 2]{LLRSS20} because of \cite[p.377, Corollary 5]{SPT95}. However we do not intend to repeat the statement in \cite{SPT95}, and choose a somewhat different way to introduce \emph{almost periodic functions with jump discontinuities and $\delta$-interactions}. The first difference is to take into account the effect of $\delta$-interactions. The second one is to restrict the locations of the (possible) discontinuity points and $\delta$-interactions to \emph{almost periodic point sets}, which have already been defined in \cite{ZZ11,DZ}. The third one is to choose a discrete framework to define the almost periodic functions with $\delta$-interactions from the point of view of topology, where a base for some uniformity will be constructed and the validity of compactness statements will be used; see \cite{Ke55}. 

%
%
%
%
%

\subsection{Point Sets}

We restrict to point sets in one dimensional case, and then recall the notion of almost periodic point sets defined in \cite{ZZ11, DZ}. It should be mentioned that in order to describe Delone dynamical systems, Lenz and Stollmann \cite{LS05} introduced the notion of almost periodic point sets in $\R^d$. For the general case that is defined on locally compact abelian groups, see \cite{KL13, LS19}.

We assume that $\Gamma =\{x_i\}_{i \in \Z} \in L(\Z)$ satisfies the following requirements
    \[
    x_0=0 \qquad \mbox{and}\qquad 0 < \inf_{i \in \Z} \Delta x_i \le \sup_{i \in \Z} \Delta x_i < \oo, \, \mbox{where } \Delta x_i := x_i - x_{i-1}.
    \]
Note that the first requirement is natural because we may translate point sets such that the zero point is included. The second one is an indispensable condition in order to introduce the recurrence property of point sets encoded in the notion of almost periodicity, because we need exclude point sets with finite limit points. Denote by $L_0(\Z)$ the space of such all point sets in $\R$. It can be equipped with the metric 
    \be \label{dist-11}
    \dist (\Ga_1,\Ga_2) := \max \left\{ \wds(\Ga_1, \Ga_2), \ \wds(\Ga_2, \Ga_1) \right\},
    \ee
where
    \[
    \wds(\Ga_1,\Ga_2) := \disp \sup_{i \in \Z} \min_{j \in \Z} |x^1_i -  x^2_j|, \quad \mbox{for } \Ga_k = \{ x^k_i \}_{i \in \Z} \in L_0(\Z), \ k=1,2.
    \]
The metric $\dist(\cdot,\cdot)$ may be regarded as the Hausdorff metric. Note that the space $(L_0(\Z),\dist)$ is not complete. However, given any $0 < m \le M < \oo$, the set
    \be \lb{lmm}
    L_{m,M}(\Z):= \Bigl\{ \Ga = \{ x_i \}_{i \in \Z} \in L_0(\Z): \Delta x_i \in [m,M], \ \forall\ {i \in \Z} \Bigr\}
    \ee
is a closed subset of $L_0(\Z)$ and it is obvious that
    \[
    L_0(\Z) = \bigcup_{0 < m \le M < \oo} L_{m,M}(\Z).
    \]
Furthermore we have

\begin{lemma} \cite{ZZ11,DZ} \lb{la-com}
$(L_{m,M}(\Z),\dist)$ is a complete space.
\end{lemma}

\begin{lemma} \cite{ZZ11,DZ} \lb{dist-l} For $\Ga_k = \{ x^k_i \} \in L_{m,M}(\Z), k=1,2$, we have:

\begin{itemize}
\smallskip
\item[{\romannumeral1}):] $\dist(\Ga_1, \Ga_2) \le M/2$; and
\smallskip
\item[{\romannumeral2}):] if $\dist(\Ga_1,\Ga_2) < m/2$, then
    \be \lb{dist-s}
    \dist(\Ga_1,\Ga_2) = \sup_{i \in \Z} |x^1_i - x^2_{i}|.
    \ee
\end{itemize}

\end{lemma}

By \x{dist-s}, the convergence in $(L_{m,M}(\Z),\dist)$ can be characterized in the following way.

\begin{lemma} \cite{ZZ11,DZ} \label{con-l}
Let $\Ga_k = \{ x_i^k \}_{i \in \Z} \in L_{m,M}(\Z)$, $k \in \N_0$. Then
$$\disp\lim_{k \to +\infty}\dist(\Ga_k,\Ga_0)= 0$$
if and only if
    \[
    \lim_{k \to \infty} \sup_{i \in \Z} |x_{i}^k - x_{i}^0| = 0.
    \]
\end{lemma}

We define the shift on $L_{0}(\Z)$ as in \cite{DZ}. For $\Gamma \in L_{0}(\Z)$ and $\tau \in \Z$, the shift of $\Gamma$ is
    \be \lb{sft-l}
    \Gamma \cdot \tau := \{ \hat x_i \}_{i \in \Z} \in L_{0}(\Z), \qq  \hat x_i := x_{i+\tau} - x_\tau.
    \ee
The family of shifts $\{ \Gamma \cdot \tau \}_{\tau \in \Z}$ yields a dynamical system on $L_{0}(\Z)$ with the following property: 
    \be \lb{dy-l}
    \Gamma \cdot (\tau_1 + \tau_2) = (\Gamma \cdot \tau_1) \cdot \tau_2 \quad \mbox{for } \tau_1, \tau_2 \in \Z.
    \ee
Note that this is not an isometric system. But we have the equicontinuity condition \x{we-is}, because we have

\begin{lemma} \lb{la-tau} \cite[Lemma 2.13]{DZ}
Let $\Gamma_k \in L_{m,M}(\Z),\ k=1,2$, and $\dist(\Gamma_1,\Gamma_2) < m/2$. Then for all $\tau \in \Z$, we have
    \[
    \dist(\Gamma_1 \cdot \tau,\Gamma_2 \cdot \tau) \le 2 \dist(\Gamma_1, \Gamma_2).
    \]
\end{lemma}

Obviously, by Lemma \ref{la-tau} and Lemma \ref{la-com}, we have
    \[
    \Bigl(L_{m,M}(\Z)\Bigr)_{ec}=L_{m,M}(\Z),\q \mbox{and}\q \Bigl(L_{m,M}(\Z)\Bigr)_{co}=L_{m,M}(\Z).
    \]
Take $(Y,\dist):=(L_{m,M}(\Z),\dist)$. Then Definition \ref{un-ap} gives a characterization of \emph{almost periodic point sets}. Denote by $L_{m,M,ap}(\Z)$ the space of all almost periodic point sets in $L_{m,M}(\Z)$. An example of an almost periodic point set is 
    \[
    \Ga_a := \{ i + a \sin i \}_{i \in \Z},
    \]
where $|a| < 1$; see \cite{ZZ11,DZ}.

\begin{remark}
{\rm The definition of almost periodic point sets is equivalent to that of almost periodic lattices as defined in \cite{ZZ11} in which the parameter $\tau$ runs over the real axis and a $\R$ action on $L_{m,M}(\Z)$ is involved.
}
\end{remark}

By Lemma \ref{ap-com} we have

\begin{lemma}
$(L_{m,M,ap}(\Z),\dist)$ is a complete space.
\end{lemma}

Similar to the mean values of almost periodic points, we may introduce the following quantity for almost periodic point sets.

\begin{lemma} \cite{ZZ11,DZ} \lb{den-l}
Let $\Gamma \in L_{m,M,ap}(\Z)$. Then the limit
    \[
    \lim_{z_2-z_1 \to +\oo} \frac{1}{z_2-z_1} \#\Bigl(\ti{\Gamma} \cap [z_1,z_2)\Bigr) =: [\Gamma]  \in \left[ \frac{1}{M}, \frac{1}{m} \right]
    \]
exists uniformly for all $\ti{\Gamma} \in \mathrm{H}(\Gamma)$, where $\#(\cdot)$ is the function counting the number of elements in a set. We call $[\Gamma]$ the density of $\Gamma$.
\end{lemma}

%
%
%
%
%
\subsection{Uniform Topology}

Based on the results above, we consider the subspace of $PC_b(\R)$ consisting of functions with jump discontinuities at points of $\Gamma \in L_{0}(\Z)$, and denote it by $PC_{b,0}(\R)$. We use the pair $(f,\Gamma)$ to represent an element in $PC_{b,0}(\R)$. Similarly, $PC_{u,0}(\R)$ denotes the subspace of $PC_{b,0}(\R)$ consisting of functions that are uniformly continuous on $\R \setminus \Gamma$, where $\Gamma \in L_0(\Z)$. Let $f$ be an even function and 
    \[
    f(x):=\left\lbrace \begin{array}{lll}
    x, \q \AND x \in (0,1), \\
    0, \q \AND x \in (i,i+1-\frac{1}{i+1}], \\
    (i+1)(x-i-1)+1, \q \AND x \in [i+1-\frac{1}{i+1},i+1),
    \end{array} \right.
    \]
where $i \in \N$. Then $f \in PC_{b,0}(\R)$ and $f$ is uniformly continuous on each interval of continuity from $\R \setminus \Z$, but $f \not \in PC_{u,0}(\R)$. This example also shows that $f$ is uniformly continuous on $\R \setminus F_r(\Z)$, for all $r >0$, where $F_r(\Z)$ is defined by \x{fr}. When the effect of $\delta$-interactions is taken into account, we denote by $PC_{b,\delta,0}(\R)$ the space of all bounded and piecewise continuous functions with jump discontinuities and $\delta$-interactions at points of $\Gamma \in L_{0}(\Z)$. Here boundedness means that a piecewise continuous function with no $\delta$-interaction is bounded on $\R$. Similarly $PC_{u,\delta,0}(\R)$ denotes the subspace of $PC_{b,\delta,0}(\R)$ consisting of all functions that are uniformly continuous on $\R \setminus \Gamma$, where $\Gamma \in L_0(\Z)$. The triple $(f,V,\Gamma)$ represents an element \x{pcde} in $PC_{u,\delta,0}(\R)$. For simplicity, we adopt the notation
$$\ce{^{V}_{\Gamma\,}}f:=(f,V,\Gamma).$$
If the point set is restricted in $L_{m,M}(\Z)$, then we denote the subspace by $PC_{u,\delta,m,M}(\R)$. Obviously we have
    \[
    \disp PC_{u,\delta,0}(\R) = \bigcup_{0 < m \le M < \oo} PC_{u,\delta,m,M}(\R).
    \]
We equip $PC_{u,\delta,m,M}(\R)$ with the uniform topology as follows. For $r>0$, a closed $r$-neighborhood of a subset $A \subset \R$ is denoted by
    \be \lb{fr}
    F_r(A):=\{x\in \R: |x-y|\le r\mbox{ for some } y \in A \}.
    \ee
Introduce the family $\{S_r\}_{r>0}$ of subsets of the product space $X:=PC_{u,\delta,m,M}(\R) \times PC_{u,\delta,m,M}(\R)$ by
    \bea
    \lb{def-s} S_r \hh & :=  &\hh \Bigl\{\Bigl(\ce{^{V_1}_{\Gamma_1}}f_1,\ce{^{V_2}_{\Gamma_2}}f_2\Bigr) \in X: \dist(\Gamma_1,\Gamma_2)<r, \ \|V_1-V_2\|_{\infty}<r \mbox{~and~} \\
    \EM |f_1(x)-f_2(x)|<r, \, \forall \, x\in \R\setminus F_r(\Gamma_1 \cup \Gamma_2)\Bigr\}  \subset X. \nn
    \eea
The set of all pairs $\Bigl(\ce{^{V}_{\Gamma\,}}f, \ce{^{V}_{\Gamma\,}}f\Bigr)$ for $\ce{^{V}_{\Gamma\,}}f \in PC_{u,\delta,m,M}(\R)$ is called the \emph{diagonal}, and is denoted by $\Delta\Bigl(PC_{u,\delta,m,M}(\R)\Bigr)$. For $r_i>0,\ i=1,2$, the \emph{composition} $S_{r_1} \circ S_{r_2}$ denotes the set of all pairs $\Bigl(\ce{^{V_1}_{\Gamma_1}}f_1,\ce{^{V_3}_{\Gamma_3}}f_3\Bigr)$ such that one has $\Bigl(\ce{^{V_1}_{\Gamma_1}}f_1,\ce{^{V_2}_{\Gamma_2}}f_2\Bigr) \in S_{r_1}$ and $\Bigl(\ce{^{V_2}_{\Gamma_2}}f_2,\ce{^{V_3}_{\Gamma_3}}f_3\Bigr) \in S_{r_2}$ for some $\ce{^{V_2}_{\Gamma_2}}f_2 \in PC_{u,\delta,m,M}(\R)$. Let
    \be \lb{con-u}
    U_0:=X, \ U_n:=S_{4^{-n}},\ n \in \N, \q \mbox{and} \q \mathscr{U}:=\{U_n\}_{n \in \N_0}.
    \ee
Then we have

\begin{lemma} \lb{lm-sf} For the subfamily $\mathscr{U}$, we have

\begin{itemize}
\smallskip
\item[{\romannumeral1}):] for all $n \in \N_0$, $\Delta\Bigl(PC_{u,\delta,m,M}(\R)\Bigr) \subset U_n$;
\smallskip
\item[{\romannumeral2}):] if $\Bigl(\ce{^{V_1}_{\Gamma_1}}f_1,\ce{^{V_2}_{\Gamma_2}}f_2\Bigr) \in U_n$, then $\Bigl(\ce{^{V_2}_{\Gamma_2}}f_2,\ce{^{V_1}_{\Gamma_1}}f_1\Bigr) \in U_n$;
\smallskip
\item[{\romannumeral3}):] if $n_1<n_2$, then $U_{n_2} \subset U_{n_1}$;
\smallskip
\item[{\romannumeral4}):]
$\bigcap_{n \in \N_0} U_{n}= \Delta\Bigl(PC_{u,\delta,m,M}(\R)\Bigr)$; and
\smallskip
\item[{\romannumeral5}):] for any $U_n$, there exists some $U_{\tilde{n}}$ such that $U_{\tilde{n}} \circ U_{\tilde{n}} \subset U_n$.

\end{itemize}
\end{lemma}

\begin{proof} \fbox{{\rm {\romannumeral1}), {\romannumeral2}), {\romannumeral3}) and {\romannumeral4})}}\ : These statements are obvious.

\fbox{{\rm {\romannumeral5})}}\ : It suffices to show that for
any $r < m/2$, we have $S_{r/2} \circ S_{r/2} \subset S_r$. Indeed, let $\Bigl(\ce{^{V_1}_{\Gamma_1}}f_1,\ce{^{V_2}_{\Gamma_2}}f_2\Bigr) \in S_{r/2}$ and $\Bigl(\ce{^{V_2}_{\Gamma_2}}f_2,\ce{^{V_3}_{\Gamma_3}}f_3\Bigr) \in S_{r/2}$. Due to Lemma \ref{dist-l} ${\rm {\romannumeral2})}$ and \x{def-s}, we have
    \[
    F_{r/2}(\Gamma_1 \cup \Gamma_2) \subset F_{r}(\Gamma_1) \q\mbox{~and~}\q F_{r/2}(\Gamma_2 \cup \Gamma_3) \subset F_{r}(\Gamma_3).
    \]
This implies that
    \[
    F_{r/2}(\Gamma_1 \cup \Gamma_2) \cup F_{r/2}(\Gamma_2 \cup \Gamma_3) \subset F_{r}(\Gamma_1) \cup F_{r}(\Gamma_3)=F_{r}(\Gamma_1 \cup \Gamma_3).
    \]
Then for $x\in \R\setminus F_r(\Gamma_1 \cup \Gamma_3)$, we have
    \[
    |f_1(x)-f_3(x)|\le |f_1(x)-f_2(x)|+|f_2(x)-f_3(x)|<r.
    \]
Thus we have the desired result $\Bigl(\ce{^{V_1}_{\Gamma_1}}f_1,\ce{^{V_3}_{\Gamma_3}}f_3\Bigr) \in S_r$.
\end{proof}

By \cite[p.177, Theorem 2]{Ke55}, we know that $\mathscr{U}$ is a base for some uniformity for $PC_{u,\delta,m,M}(\R)$. Then $PC_{u,\delta,m,M}(\R)$ can be equipped with the uniform topology denoted by $\mathscr{T}$ that is generated from $\mathscr{U}$.

\begin{theorem}{\rm (\textbf{Alexandroff-Urysohn})} \cite[p.186, Theorem 13]{Ke55} \lb{thm-au} A uniform space $(Y, \mathscr{T})$ is metrizable if and only if it is Hausdorff and its uniformity $\mathscr{U}$ has a countable base. Furthermore, let $\mathscr{U}=\{U_n\}_{n \in \N_0}$ be the base and $\dist: Y \times Y \to \R^+_0$ be the induced metric. Then we have
    \be \lb{u-d}
    U_n \subset \{(x,y) \in Y \times Y:\ \dist(x,y)<2^{-n}\} \subset U_{n-1}, \qqf  n \in \N.
    \ee
\end{theorem}

\begin{lemma} \lb{di-pcud}
The uniform space $(PC_{u,\delta,m,M}(\R),\mathscr{T})$ is metrizable.
\end{lemma}

\begin{proof} By Lemma \ref{lm-sf} {\romannumeral4}), we know that $(PC_{u,\delta,m,M}(\R),\mathscr{T})$ is Hausdorff. By construction \x{con-u}, the uniformity $\mathscr{U}$ has a countable base. Due to Theorem \ref{thm-au}, we have the desired result.
\end{proof}

Then we may construct a metric $\dist: PC_{u,\delta,m,M}(\R) \times PC_{u,\delta,m,M}(\R) \to \R^+_0$  such that
    \be \lb{un-dist}
    (PC_{u,\delta,m,M}(\R),\dist)=(PC_{u,\delta,m,M}(\R),\mathscr{T}). 
    \ee
The convergence in $(PC_{u,\delta,m,M}(\R),\dist)$ can be characterized in the following way.

\begin{lemma} \lb{tr-con} 
Let $\ce{^{V_k}_{\Gamma_k}}f_k \in PC_{u,\delta,m,M}(\R)$, $k \in \N_0$. Then $\disp \lim_{k\to +\infty}\dist\Bigl(\ce{^{V_k}_{\Gamma_k}}f_k,\ce{^{V_0}_{\Gamma_0}}f_0\Bigr)= 0$ if and only if for any $\e>0$ there exists $k_{\e} \in \N$ such that for all $k \ge k_{\e}$, we have
    \[
    \dist(\Gamma_k,\Gamma_0)<\e, \  \|V_k-V_0\|_{\infty}<\e, \mbox{~and~}  |f_k(x)-f_0(x)|<\e, \, \forall \, x\in \R\setminus F_{\e}(\Gamma_k \cup \Gamma_0).
    \]
\end{lemma}

\begin{proof}
Due to \x{u-d}, \x{con-u} and \x{def-s}, it is easy to check the characterization.
\end{proof}

\begin{remark} \lb{re-f}
{\rm $F_{\e}(\Gamma_k \cup \Gamma_0)$ can be replaced by $F_{\e}(\Gamma_0)$. Since $F_{\e}(\Gamma_0) \subset F_{\e}(\Gamma_k \cup \Gamma_0)$, we only need to show this implication $\Longrightarrow$. Indeed, assume that
    \[
    \disp\lim_{k \to +\infty}\dist\Bigl(\ce{^{V_k}_{\Gamma_k}}f_k,\ce{^{V_0}_{\Gamma_0}}f_0\Bigr)= 0.
    \]
For any $\e>0$, there exists $k_{\e/2} \in \N$ such that
    \[
    \dist(\Gamma_k,\Gamma_0)<\e/2, \qqf k \ge k_{\e/2}.
    \]
Without loss of generality, let $\e<m$. It follows from \x{dist-s} that
    \[
    F_{\e/2}(\Gamma_k \cup \Gamma_0)= F_{\e/2}(\Gamma_k) \cup F_{\e/2}(\Gamma_0) \subset F_{\e}(\Gamma_0), \qqf k \ge k_{\e/2}.
    \]
Thus we have the desired assertion.
}
\end{remark}

%
%
%
%
%

\subsection{Functions With Jump Discontinuities and $\delta$-Interactions}

In this subsection, we introduce the almost periodic functions with jump discontinuities and $\delta$-interactions. We consider a $\Z$ action on $PC_{u,\delta,m,M}(\R)$ by shifts. Let $f: \R \to \K$ be a piecewise continuous function and $\tau \in \R$. The shift of $f$ is
    \[
    f \cdot \tau := f(\cdot+ \tau).
    \]
Then denote for $\ce{^{V}_{\Gamma}}f \in PC_{u,\delta,m,M}(\R)$ and $\tau \in \Z$ the corresponding shifted element in $PC_{u,\delta,m,M}(\R)$ by
    \be \lb{sh-tr}
    \ce{^{V}_{\Gamma\,}}f \cdot \tau:= \ce{^{V\cdot \tau}_{\Gamma\cdot \tau}}(f\cdot x_\tau),
    \ee
where $\Gamma\cdot \tau$ and $V\cdot \tau$ are defined by \x{sft-l} and \x{sft-v}, respectively. Obviously we have
    \[
    \Orb\Bigl(\ce{^{V}_{\Gamma\,}}f\Bigr) \subset PC_{u,\delta,m,M}(\R).
    \]
Due to \x{lmm} and \x{dy-l}, the family of shifts $\Bigl\{\ce{^{V}_{\Gamma\,}}f \cdot \tau \Bigr\}_{\tau \in \Z}$ yields a dynamical system on $PC_{u,\delta,m,M}(\R)$ with the following property 
    \be \lb{dy-tr}
    \ce{^{V}_{\Gamma\,}}f \cdot (\tau_1 + \tau_2) = \Bigl(\ce{^{V}_{\Gamma\,}}f \cdot \tau_1\Bigr) \cdot \tau_2, \quad \mbox{for } \tau_1, \tau_2 \in \Z.
    \ee
This is not an isometric system because of shifts of point sets. However we have

\begin{lemma} \lb{tr-tau}
The shift action on $(PC_{u,\delta,m,M}(\R),\dist)$ satisfies the equicontinuity condition \xx{we-is}.
\end{lemma} 

\begin{proof}
Let $\ce{^{V_k}_{\Gamma_k}}f_k \in PC_{u,\delta,m,M}(\R), \ k=1,2$. It suffices to show that for any $n \in \N$, there exists $k_n \in \N$ such that if $\dist\Bigl(\ce{^{V_1}_{\Gamma_1}}f_1,\ce{^{V_2}_{\Gamma_2}}f_2\Bigr)<2^{-k_n}$, then we have
    \be \lb{tr-tau0}
    \dist\Bigl(\ce{^{V_1}_{\Gamma_1}}f_1 \cdot \tau,\ce{^{V_2}_{\Gamma_2}}f_2 \cdot \tau\Bigr) < 2^{-n}, \qqf \tau \in \Z.
    \ee
Denote $\Gamma_k:=\{x^k_i\}_{i \in \Z} \in L_{m,M}(\Z),\ k=1,2$. Due to \x{u-d}, \x{con-u}, \x{def-s} and \x{sh-tr}, it suffices to prove that
    \be \lb{tr-tau1}
    \dist(\Gamma_1 \cdot \tau,\Gamma_2 \cdot \tau) < 4^{-n}, \qqf \tau \in \Z,
    \ee
    \be \lb{tr-tau2}
    \|V_1 \cdot \tau -V_2 \cdot \tau \|_{\infty} < 4^{-n}, \qqf \tau \in \Z,
    \ee
and
    \be \lb{tr-tau3}
    |f_1\cdot x_{\tau}^1(x)-f_2\cdot x_{\tau}^2(x)| < 4^{-n} \ \mbox{for~} x \in \R \setminus F_{4^{-n}}(\Gamma_1 \cdot \tau \cup \Gamma_2 \cdot \tau), \qqf \tau \in \Z.
    \ee
Since $\ce{^{V_2}_{\Gamma_2}}f_2 \in PC_{u,\delta,m,M}(\R)$, for any $n \in \N$, there exists $\ti k_n \in \N$ such that
    \bea
    \EM \lb{tr-tau4}  |f_2(\tilde{x})-f_2(\check{x})|< 4^{-n-1}, \\
    \EM \mbox{for~} |\tilde{x}-\check{x}|<4^{-\ti k_n}  \mbox{~and~} \tilde{x},\ \check{x} \mbox{~belong~to~the~same~interval~from~} \R \setminus \Gamma_2. \nn
    \eea
Without loss of generality, assume that $\ti k_n \ge n$. We assert that
    \be \lb{tr-tau45}
    k_n:=\ti k_n+2 \ge n+2
    \ee
is the desired number such that \x{tr-tau0} holds. Indeed, if $\dist\Bigl(\ce{^{V_1}_{\Gamma_1}}f_1,\ce{^{V_2}_{\Gamma_2}}f_2\Bigr)<2^{-k_n}$, then from \x{u-d}, \x{con-u} and \x{def-s} we find
    \be \lb{tr-tau6}
    \dist(\Gamma_1,\Gamma_2) < 4^{-k_n+1},
    \ee
    \be \lb{tr-tau7}
    \|V_1-V_2\|_{\infty} < 4^{-k_n+1},
    \ee
and
    \be \lb{tr-tau8}
    |f_1(x)-f_2(x)| < 4^{-k_n+1} \quad \mbox{for~} x \in \R \setminus F_{4^{-k_n+1}}(\Gamma_1 \cup \Gamma_2).
    \ee
By Lemma \ref{la-tau}, \x{tr-tau6} and \x{tr-tau45}, we have the desired result \x{tr-tau1}. By \x{bs-is}, \x{tr-tau7} and \x{tr-tau45}, we have the desired result \x{tr-tau2}. Denote $\delta_{\tau}:=x_{\tau}^2-x_{\tau}^1, \ \tau \in \Z$. From \x{dist-s} and \x{tr-tau6} we see that
    \be \lb{tr-tau89}
    |\delta_{\tau}|<4^{-k_n+1}, \qqf \tau \in \Z.
    \ee
This implies that
    \be \lb{tr-tau9}
    F_{4^{-k_n+1}}(\{x^2_{i}-x^1_\tau\}_{i \in \Z})= F_{4^{-k_n+1}}(\{x^2_{i+\tau}-x^2_\tau+\delta_\tau\}_{i \in \Z}) \subset F_{4^{-k_n+2}}(\Gamma_2 \cdot \tau).
    \ee
We make the following claims.

\medskip
\fbox{Claim 1}\ : for any $\tau \in \Z$, we have
    \be \lb{tr-tau10}
    |f_1 \cdot x_{\tau}^1(x)-f_2 \cdot x_{\tau}^1(x)| < 4^{-n-1}, \qqf x \in \R \setminus F_{4^{-n}}(\Gamma_1 \cdot \tau \cup \Gamma_2 \cdot \tau).
    \ee
In fact, by \x{tr-tau9} and \x{tr-tau45}, we have
    \be \lb{tr-tau11}
    F_{4^{-k_n+1}}(\{x^1_{i}-x^1_\tau\}_{i \in \Z} \cup \{x^2_{i}-x^1_\tau\}_{i \in \Z}) \subset F_{4^{-n}}(\Gamma_1 \cdot \tau \cup \Gamma_2 \cdot \tau).
    \ee 
For any $x \in \R \setminus F_{4^{-n}}(\Gamma_1 \cdot \tau \cup \Gamma_2 \cdot \tau)$, \x{tr-tau11} yields
    \[
    x+x_{\tau}^1 \in  \R \setminus F_{4^{-k_n+1}}(\Gamma_1 \cup \Gamma_2).
    \]
The claim \x{tr-tau10} is deduced by \x{tr-tau8} and \x{tr-tau45}.

\medskip
\fbox{Claim 2}\ : for any $\tau \in \Z$, we have
    \be \lb{tr-tau12}
    |f_2 \cdot x_{\tau}^1(x)-f_2 \cdot x_{\tau}^2(x)| < 4^{-n-1}, \qqf x \in \R \setminus F_{4^{-n}}(\Gamma_2 \cdot \tau).
    \ee
In fact, by \x{tr-tau9} and \x{tr-tau45} we have
    \be \lb{tr-tau13}
    F_{4^{-k_n+1}}(\{x^2_{i}-x^2_\tau\}_{i \in \Z} \cup \{x^2_{i}-x^1_\tau\}_{i \in \Z}) \subset F_{4^{-n}}(\Gamma_2 \cdot \tau)
    \ee 
For any $x\in \R \setminus F_{4^{-n}}(\Gamma_2 \cdot \tau)$, \x{tr-tau13} yields
    \[
    x+x_{\tau}^1,\ x+x_{\tau}^2 \in  \R \setminus F_{4^{-k_n+1}}(\Gamma_2) \subset \R \setminus \Gamma_2.
    \]
It follows from \x{tr-tau89} that $x+x_{\tau}^1,\ x+x_{\tau}^2$ necessarily belong to the same interval from $\R \setminus \Gamma_2$. The claim \x{tr-tau12} is deduced by \x{tr-tau45} and \x{tr-tau4}.

We obtain the desired result \x{tr-tau3} by the two claims above, completing the proof.
\end{proof}

\begin{remark} \lb{pt-ec}
{\rm As a natural consequence of Lemma \ref{tr-tau}, we see that each $\ce{^{V}_{\Gamma\,}}f \in PC_{u,\delta,m,M}(\R)$ is an equicontinuous point, i.e.,
    \[
    \Bigl(PC_{u,\delta,m,M}(\R)\Bigr)_{ec}=PC_{u,\delta,m,M}(\R).
    \]
}
\end{remark}

Furthermore we obtain

\begin{lemma} \lb{pcu-com}
For any $\ce{^{V}_{\Gamma\,}}f \in PC_{u,\delta,m,M}(\R)$, $\ce{^{V}_{\Gamma\,}}f$ is a complete point in $PC_{u,\delta,m,M}(\R)$. That is,
    \[
    \Bigl(PC_{u,\delta,m,M}(\R)\Bigr)_{co}=PC_{u,\delta,m,M}(\R).
    \]
\end{lemma}

\begin{proof}
It suffices to show that each Cauchy sequence in $\Bigl(\mathrm{H}\Bigl(\ce{^{V}_{\Gamma\,}}f\Bigr),\dist\Bigr)$ converges to a point in $\mathrm{H}\Bigl(\ce{^{V}_{\Gamma\,}}f\Bigr)$. Let $\Bigl\{\ce{^{V_k}_{\Gamma_k}}f_k\Bigr\}_{k\in \N} \subset \mathrm{H}\Bigl(\ce{^{V}_{\Gamma\,}}f\Bigr)$ be a Cauchy sequence. Then for each $\ce{^{V_k}_{\Gamma_k}}f_k$, there exists $\ce{^{V_{n_k}}_{\Gamma_{n_k}}}f_{n_k}:= \ce{^{V}_{\Gamma\,}}f \cdot n_k \in \Orb\Bigl(\ce{^{V}_{\Gamma\,}}f\Bigr)$ such that
    \be \lb{puc-com0}
    \dist\Bigl(\ce{^{V_k}_{\Gamma_k}}f_k, \ce{^{V_{n_k}}_{\Gamma_{n_k}}}f_{n_k}\Bigr)<1/k.
    \ee
Note that by \x{sh-tr}, we have
    \be \lb{puc-com01}
    f_{n_k}= f \cdot x_{n_k}\q \mbox{and}\q \Gamma_{n_k}=\Gamma \cdot n_k,
    \ee
where $\Gamma=\{x_i\}_{i \in \Z}$. It follows from \x{u-d}, \x{con-u} and \x{def-s} that both $\{\Gamma_{n_k}\}$ and $\{V_{n_k}\}$ are Cauchy sequences in $(L_{m,M}(\Z),\dist)$ and $(\ell^{\oo}(\Z),\dist)$, respectively. Due to Lemma \ref{la-com} and Lemma \ref{bs-com}, there exist $\Gamma_0=\{x^0_i\}_{i \in \Z} \in L_{m,M}(\Z)$ and $V_0 \in \ell^{\oo}(\Z)$ such that
    \be \lb{puc-com1}
    \lim_{k \to +\infty}\dist(\Gamma_{n_k},\Gamma_0)=0,
    \ee
and
    \be \lb{puc-com2}
    \lim_{k \to +\infty}\|V_{n_k}-V_0\|_{\infty}=0.
    \ee
Again by \x{u-d}, \x{con-u} and \x{def-s}, we know that for any $\e>0$, there exists $k_{\e} \in \N$ such that
    \be \lb{puc-com3}
    |f_{n_{k_1}}(x)-f_{n_{k_2}}(x)|< \e, \mbox{~for~all~} x\in \R\setminus F_{\e}(\Gamma_{n_{k_1}} \cup \Gamma_{n_{k_2}}) \mbox{~and~} k_1, \ k_2 \ge k_{\e}.
    \ee
We make the following claims.

\medskip
\fbox{Claim 1}\ : for any $x\in  \R \setminus \Gamma_0$, the sequence $\{f_{n_k}(x)\}$ converges to a point in $\R$ denoted by $f_0(x)$. Indeed, let $d:=\min_{i \in \Z}|x-x^0_i|>0$. Without loss of generality, assume that
    \be \lb{puc-com34}
    \e<d.
    \ee
By \x{puc-com1} and \x{puc-com3}, for any $\e>0$, there exists a common $\ti k_{\e} \in \N$ such that
    \be \lb{puc-com4}
    \dist(\Gamma_{n_k},\Gamma_0)<\e/4, \qqf k \ge \ti k_{\e},
    \ee
and
    \be \lb{puc-com5}
    |f_{n_{k_1}}(x)-f_{n_{k_2}}(x)|< \e/4, \mbox{~for~all~} x\in \R\setminus F_{\e/4}(\Gamma_{n_{k_1}} \cup \Gamma_{n_{k_2}}) \mbox{~and~} k_1,\ k_2 \ge \ti k_{\e}.
    \ee
When $x\in  \R \setminus \Gamma_0$ and $k_1,\ k_2 \ge k_{\e}$, by \x{puc-com34} and \x{puc-com4}, we have $x \in \R\setminus F_{\e/4}(\Gamma_{n_{k_1}} \cup \Gamma_{n_{k_2}})$. This implies from \x{puc-com5} that $\{f_{n_k}(x)\}$ is a Cauchy sequence. Thus we have the desired result.

\medskip
\fbox{Claim 2}\ : $f_0(x)$ is bounded on $\R\setminus \Gamma_0$. Indeed, because of \x{puc-com01} and the boundedness of $f$ on $\R\setminus \Gamma$, the family $\{f_{n_k}: \R\setminus \Gamma_{n_k} \to \K \}_{k \in \N}$ is uniformly bounded. Due to the claim above, we have
    \[
    \lim_{k \to +\infty} f_{n_k}(x)= f_0(x), \qqf x \in \R\setminus \Gamma_0.
    \]
This implies the boundedness of $f_0(x)$ on $\R\setminus \Gamma_0$.

\medskip
\fbox{Claim 3}\ : for any $\e>0$, there exists $k_{\e} \in \N$ such that for all $k \ge k_{\e}$, we have
    \be \lb{puc-com8}
    |f_{n_k}(x)-f_0(x)|<\e, \, \qqf x\in \R\setminus F_{\e}(\Gamma_{n_k} \cup \Gamma_0).
    \ee
Indeed, by \x{puc-com4} and \x{puc-com5}, for $k_1,\ k_2 \ge \ti k_{\e}$, we have
    \[
    F_{\e/4}(\Gamma_{n_{k_1}} \cup \Gamma_{n_{k_2}}) \subset F_{\e/2}(\Gamma_0).
    \]
This implies that
    \[
    |f_{n_{k_1}}(x)-f_{n_{k_2}}(x)|< \e/4, \qqf x\in \R\setminus F_{\e/2}(\Gamma_0).
    \]
Let $k_2 \nearrow +\infty$. For $k_1 \ge \ti k_{\e}$, we have
    \be \lb{puc-com6}
    |f_{n_{k_1}}(x)-f_{0}(x)| \le \e/4, \qqf x\in \R\setminus F_{\e/2}(\Gamma_0).
    \ee
Claim \x{puc-com8} is deduced by Remark \ref{re-f} and \x{puc-com6}.

\medskip
\fbox{Claim 4}\ : $f_0(x)$ is uniformly continuous on $\R \setminus \Gamma_0$. Indeed, denote
    \[
    d_k:=\dist(\Gamma_{n_k},\Gamma_0), \q \mbox{and}\q A_{k}:=\R \setminus F_{2d_k}(\Gamma_0).
    \]
By \x{puc-com1}, we know that $d_k$ is monotonically decreasing to zero, and
    \be \lb{puc-com67}
    A_k \subset A_{k+1}, \q \mbox{and}\q \bigcup_{k \in \N} A_k=\R \setminus \Gamma_0.
    \ee
Since $\ce{^{V}_{\Gamma}}f \in PC_{u,\delta,m,M}(\R)$ and $\ce{^{V_{n_k}}_{\Gamma_{n_k}}}f_{n_k} \in \Orb\Bigl(\ce{^{V}_{\Gamma\,}}f\Bigr)$, we know that
    \[
    \mathfrak{F}:=\{(f_{n_k},\Gamma_{n_k})\}_{k \in \N} \subset PC_{u,m.M}(\R)
    \]
is equicontinuous, i.e., for any $\e>0$, there exists a common $\delta_\e$ such that for any $(f_{n_i},\Gamma_{n_i}) \in \mathfrak{F}$, we have
    \bea
    \EM \lb{puc-com69}  |f_{n_i}(x_1)-f_{n_i}(x_2)|<\e/3, \\
    \EM \mbox{for~}|x_1-x_2| <\delta_\e  \mbox{~and~} x_1,\ x_2 \mbox{~belong~to~the~same~interval~from~} \R \setminus \Gamma_{n_i}. \nn
    \eea
Assume that $x_1,\ x_2$ belong to the same interval from $\R \setminus \Gamma_0$. Then by \x{puc-com67} there exists $k_{x_1,x_2} \in \N$ such that $x_1,\ x_2$ that belong to the same interval from $A_k$ for any $k \ge k_{x_1,x_2}$. Restricting $x \in A_{k_{x_1,x_2}}$, we know from \x{puc-com6} that for any $\e>0$, there exists $k_\e \in \N$ such that we have
    \be \lb{puc-com68}
    |f_{n_i}(x)-f_0(x)| < \e/3, \q \mbox{for~}i \ge k_\e.
    \ee
Note that $k_\e$ also depends on the choice of $x_1,\ x_2$. If $|x_1-x_2|<\delta_\e$, then from \x{puc-com69} and \x{puc-com68} it follows that
    \beaa
    \EM |f_0(x_1)-f_0(x_2)| \\
    \LE \Bigl|f_0(x_1)-f_{n_{k_\e}}(x_1)\Bigr|+\Bigl|f_{n_{k_\e}}(x_1)-f_{n_{k_\e}}(x_2)\Bigr|+\Bigl|f_{n_{k_\e}}(x_2)-f_0(x_2)\Bigr| \\
    \AND < \e/3+\e/3+\e/3 = \e.
    \eeaa
This means that the claim is deduced.

Based on this claim, it is easy to conclude that $f_0$ has jump discontinuities at points of $\Gamma_0$. Then we have $\ce{^{V_0}_{\Gamma_0}}f_0 \in PC_{u,\delta,m,M}(\R)$. Due to \x{puc-com1}, \x{puc-com2}, \x{puc-com8} and Lemma \ref{tr-con}, we obtain
    \be \lb{puc-com10}
    \lim_{k \to +\infty}\dist\Bigl(\ce{^{V_{n_k}}_{\Gamma_{n_k}}}f_{n_k},\ce{^{V_0}_{\Gamma_0}}f_0\Bigr) = 0.
    \ee
It follows from \x{puc-com0} and \x{puc-com10} that $\ce{^{V_0}_{\Gamma_0}}f_0 \in \mathrm{H}\Bigl(\ce{^{V}_{\Gamma\,}}f\Bigr)$, and
    \[
    \lim_{k \to +\infty}\dist\Bigl(\ce{^{V_{k}}_{\Gamma_{k}}}f_{k},\ce{^{V_0}_{\Gamma_0}}f_0\Bigr)=0.
    \]
The proof is completed.
\end{proof}

\begin{remark} \lb{re-pcu-com1}
{\rm As a byproduct of Claim 2 in the proof of Lemma \ref{pcu-com}, we may deduce that for all $\ce{^{\ti V}_{\ti \Gamma}}\ti f \in \mathrm{H}\Bigl(\ce{^{V}_{\Gamma}}f\Bigr)$, there exists $B>0$ such that
    \be \lb{unbd}
    \disp \|\ti f(x)\|_{\infty}+\|\ti V\|_{\infty} \le B.
    \ee
}
\end{remark}

\begin{remark} \lb{re-pcu-com2}
{\rm In the proof of Claim 3 in the proof of Lemma \ref{pcu-com}, the equicontinuity of $\mathfrak{F}$ is crucial. A fundamental question is whether the uniform space $(PC_{u,\delta,m,M}(\R),\dist)$ is a complete metric space. We leave it to the reader. However each uniform space is uniformly isomorphic to a dense subspace of a complete uniform space, then we may make a completion of a uniform space.
}
\end{remark}

In case there is no $\delta$-interaction, we denote by $PC_{u,0}(\R)$ the subspace of $PC_{u}(\R)$ consisting of all functions with jump discontinuities at points of $\Gamma \in L_0(\Z)$. Similarly $PC_{u,m,M}(\R)$ denotes the subspace of $PC_{u,0}(\R)$ with jump discontinuities at $\Gamma \in L_{m,M}(\Z)$. Then we have
    \[
    \disp PC_{u,0}(\R) = \bigcup_{0 < m \le M < \oo} PC_{u,m,M}(\R).
    \]
\begin{remark} \lb{rm-pcu}
{\rm All results above can be established for $PC_{u,m,M}(\R)$ in a similar way. In detail, removing the effect of $\delta$-interactions in \x{def-s}, we first construct a family still denoted by $\mathscr{U}:=\{U_n\}_{n \in \N_0}$ of subsets of the product space $PC_{u,m,M}(\R) \times PC_{u,m,M}(\R)$. Then the uniform space $(PC_{u,m,M}(\R),\mathscr{T})$ is metrizable where the topology $\mathscr{T}$ is generated from $\mathscr{U}$. Moreover we have
    \[
    \Bigl(PC_{u,m,M}(\R)\Bigr)_{ec}=PC_{u,m,M}(\R), \mbox{~and~}\Bigl(PC_{u,m,M}(\R)\Bigr)_{co}=PC_{u,m,M}(\R).
    \]
}
\end{remark}

Before introducing the main concept in this section, we recall the class of Bohr almost periodic functions.

\begin{definition} \cite{Fi74,DZ} \label{ap-f}
We say that $f \in C_b(\R)$ is Bohr almost periodic if one of the following conditions holds:
\begin{itemize}
\smallskip
\item[{\romannumeral1}):]{\rm [Bohr's definition]} for any $\eps > 0$, $\mathrm{P}_\R(f, \eps) := \left\{ \tau \in \R: \|f \cdot \tau-f\|_{\infty} < \eps \right\}$
is relatively dense in $\R$;
\smallskip
\item[{\romannumeral2}):] the hull of $f$ running over $\R$, defined by
		\be \label{hl-f}
		\mathrm{H}_{\R}(f) := \ol{\{f \cdot \tau : \tau \in \R \}}^{(C_b(\R),\|\cdot\|_{\infty})}, \nn
		\ee
		is a compact subset in $C_b(\R)$;
\smallskip		
\item[{\romannumeral3}):]{\rm [Bochner's definition]} for any sequence $\{ \ti{\tau}_k \}\subset \R$, one can extract a subsequence $\{ \tau_k \} \subset \{\ti{\tau}_k\}$ such that $\{f \cdot \tau_k\}$ is convergent in $(C_b(\R),\|\cdot\|_{\infty})$, i.e., $\{f \cdot \tau : \tau \in \R \}$ is relatively compact.
\end{itemize}
\end{definition}

\begin{remark} \lb{re-aa}
{\rm  The equivalence between conditions {\rm {\romannumeral1}) and {\romannumeral3})} in {\rm Definition \ref{ap-f}} may be regarded as the Arzel\`{a}-Ascoli theorem for Bohr almost periodic functions. That is, the family of functions $\{f \cdot \tau : \tau \in \R \}$ is relatively compact in $(C_b(\R),\|\cdot\|_{\infty})$ if and only if $\mathrm{P}_{\R}(f,\eps)$ is relatively dense for any $\e >0$. Note that the condition of relative denseness implies that $f \in C_u(\R)$. The classical Arzel\`{a}-Ascoli theorem requires that the domain of the functions is a compact Hausdorff space. Here the condition of relative denseness is to compensate for the non-compactness of $\R$.
}
\end{remark}

The difference between this definition and Definition \ref{un-ap} is that the parameter $\tau$ in this definition is required to run over the real axis. We denote the space of all Bohr almost periodic functions by $C_{ap}(\R)$. It is well known that $(C_{ap}(\R),\| \cdot \|_{\infty})$ is a Banach algebra {\rm \cite{Fi74}}. By Definition \ref{un-ap}, Remark \ref{re-co}, Lemma \ref{tr-tau} and Lemma \ref{pcu-com}, we are now in a position to introduce almost periodic functions with jump discontinuities and $\delta$-interactions.

\begin{definition} \label{ap}
$\ce{^{V}_{\Gamma\,}}f \in PC_{u,\delta,m,M}(\R)$ is called an almost periodic function with jump discontinuities and $\delta$-interactions if one of the following conditions holds:
\begin{itemize}
\smallskip		
\item[{\romannumeral1}):]{\rm [Bohr-type definition]}  for any $\e>0$, $\mathrm{P}\Bigl(\ce{^{V}_{\Gamma\,}}f, \eps\Bigr) := \left\{ \tau \in \Z: \dist\Bigl(\ce{^{V}_{\Gamma\,}}f \cdot \tau, \ce{^{V}_{\Gamma\,}}f\Bigr) < \eps \right\}$ is relatively dense in $\Z$, where $\dist$ is the metric introduced in Lemma {\rm\ref{di-pcud}};
\smallskip		
\item[{\romannumeral2}):] the hull of $\ce{^{V}_{\Gamma\,}}f$, defined by
		\be \label{hull}
		\mathrm{H}\Bigl(\ce{^{V}_{\Gamma\,}}f\Bigr)= \ol{\Bigl\{\ce{^{V}_{\Gamma\,}}f \cdot k : k \in \Z \Bigr\}}^{(PC_{u,\delta,m,M}(\R),\dist)} \nn
		\ee
		is compact;
\smallskip				
\item[{\romannumeral3}):]{\rm [Bochner-type definition]}  for any sequence $\{\ti \tau_k \}_{k \in \N}\subset \Z$, one can extract a subsequence $\{\tau_k\} \subset \{\ti \tau_k\}$ such that $\Bigl\{\ce{^{V}_{\Gamma\,}}f \cdot \tau_k\Bigr\}$ is convergent in $(PC_{u,\delta,m,M}(\R),\dist)$.
\end{itemize}
\end{definition}

We denote the space of all almost periodic functions with jump discontinuities and $\delta$-interactions at points of $\Gamma \in L_{m,M}(\Z)$ by $PC_{\delta,m,M,ap}(\R)$. As we stated in Remark \ref{rm-pcu}, Definition \ref{un-ap} can also give the characterization of \emph{almost periodic functions with jump discontinuities}. Denote the space of all almost periodic functions with only jump discontinuities at points of $\Gamma \in L_{m,M}(\Z)$ by $PC_{m,M,ap}(\R)$. An example of $PC_{m,M,ap}(\R)$ is
    \be \lb{def-f}
    f|_{(x_i,x_{i+1})}=u_i,
    \ee
where $\Gamma=\{x_i\} \in L_{m,M,ap}(\Z)$ and $\{u_i\} \in \ell_{ap}(\Z)$. An example of $PC_{\delta,m,M,ap}(\R)$ is
    \[
    \ce{^{1}_{\Gamma}}f=\ce{^{0}_{\Gamma}}f+\delta_{\Gamma},
    \]
where $\ce{^{0}_{\Gamma}}f$ is given by \x{def-f}.

Introduce the following notation:
    \be \lb{pcd0ap}
    \disp PC_{\delta,0,ap}(\R) := \bigcup_{0 < m \le M < \oo} PC_{\delta,m,M,ap}(\R),
    \ee
and
    \[
    PC_{0,ap}(\R) := \bigcup_{0 < m \le M < \oo} PC_{m,M,ap}(\R).
    \]

\begin{lemma} \lb{lm-pcdap} We have:
\begin{itemize}
\smallskip
\item[{\romannumeral1}):] $PC_{m,M,ap}(\R) \subset PC_{\delta,m,M,ap}(\R)$; $PC_{0,ap}(\R) \subset PC_{\delta,0,ap}(\R)$;
\smallskip
\item[{\romannumeral2}):] $C_{ap}(\R) \subset PC_{0,ap}(\R) \cap C(\R)$; and
\smallskip
\item[{\romannumeral3}):] if $\ce{^{V}_{\Gamma\,}}f \in PC_{\delta,m,M,ap}(\R)$, then $\ce{^{0}_{\Gamma}}f \in PC_{m,M,ap}(\R)$, $V \in \ell_{ap}(\Z)$, and $\Gamma \in  L_{m,M,ap}(\Z)$.
\end{itemize}
\end{lemma}

\begin{proof} \fbox{{\rm {\romannumeral1})}}\ : Because we may regard $(f,\Gamma)\in PC_{u,m,M}(\R)$ as an element in $PC_{u,\delta,m,M}(\R)$ with no $\delta$-interaction, {\romannumeral1}) is obvious.

\fbox{{\rm {\romannumeral2})}}\ : We assume that $f\in C_{ap}(\R)$. By Definition \ref{ap-f} {\romannumeral2}), we know that
    \[
    \mathrm{H}(f)= \ol{\{f \cdot \tau : \tau \in \Z \}}^{(C_u(\R),\|\cdot\|_{\infty})} \subset \mathrm{H}_{\R}(f)
    \]
is compact in $(C_u(\R),\|\cdot\|_{\infty})$. We regard $f$ as an element in $PC_{u,m,M}(\R)$ with no jump discontinuities and no $\delta$-interaction. Because
    \[
    (C_u(\R),\|\cdot\|_{\infty}) \hookrightarrow (PC_{u,m,M}(\R),\dist),
    \]
we have $f \in PC_{0,ap}(\R)$. The proof of {\romannumeral2}) is completed.

\fbox{{\rm {\romannumeral3})}}\ : We assume that $\ce{^{V}_{\Gamma\,}}f \in PC_{\delta,m,M,ap}(\R)$. Due to \x{u-d}, \x{con-u} and \x{def-s}, we have
    \[
    \mathrm{P}\Bigl(\ce{^{V}_{\Gamma\,}}f, \eps\Bigr) \subset \mathrm{P}\Bigl(\ce{^{0}_{\Gamma}}f, \eps\Bigr).
    \]
This implies that $\mathrm{P}\Bigl(\ce{^{0}_{\Gamma}}f, \eps\Bigr)$ is relatively dense for any $\e>0$. Thus $\ce{^{0}_{\Gamma}}f \in PC_{m,M,ap}(\R)$. Meanwhile, by Definition \ref{ap} {\romannumeral3}), Definition \ref{un-ap} {\romannumeral3}) and Lemma \ref{tr-con}, we deduce that $V \in \ell_{ap}(\Z)$ and $\Gamma \in  L_{m,M,ap}(\Z)$.  The proof of {\romannumeral3}) is completed.
\end{proof}

\begin{remark} \lb{rm-pcdap}
{\rm A natural question is whether $(PC_{\delta,m,M,ap}(\R),\dist)$ is a complete metric space. We leave it to the interested reader. Note that in case that $(PC_{u,\delta,m,M}(\R),\dist)$ is complete, we would obtain this result by Lemma \ref{ap-com}.
}
\end{remark}

%
%
%
%
%

\subsection{Mean Value}

We use Lemma \ref{atg} to introduce the \emph{mean value} of $\ce{^{V}_{\Gamma\,}}f \in PC_{\delta,0,ap}(\R)$. 

\begin{lemma} \label{mv-pcdap}
Let $\ce{^{V}_{\Gamma\,}}f \in PC_{\delta,0,ap}(\R)$. Then the limit
    \be \lb{mv-f0}
    \mathrm{M}\Bigl(\ce{^{V}_{\Gamma\,}}f\Bigr) := \lim_{z_2-z_1 \to +\oo} \frac{1}{z_2-z_1} \int_{[z_1,z_2)} \Bigl(f(x)+\sum_{i \in \Z} v_i \delta(x - x_i)\Bigr)\rd x \in \C
    \ee
exists uniformly for all $z_1,\ z_2 \in \R$.  We call it the mean value of $\prescript{V}{\Gamma\,}f$.
\end{lemma}

\begin{proof}
By \x{pcd0ap}, there exist $m, \ M >0$ such that $\ce{^{V}_{\Gamma\,}}f \in PC_{\delta,m,M,ap}(\R)$. Denote $\Gamma=\{x_i\}_{i \in \Z} \in L_{m,M}(\Z)$. Then for any $z_1, \ z_2 \in \R$, there exist $n_1, \ n_2 \in \Z$ such that
    \be \lb{mv-f1}
    x_{n_1} \le z_1 < x_{n_1+1} \mbox{~and~} x_{n_2} \le z_2 < x_{n_2+1}.
    \ee
We make the following claims.

\medskip
\fbox{Claim 1}\ : the following relation holds:
    \bea
    \EM \lim_{z_2-z_1 \to +\oo} \frac{1}{z_2-z_1} \int_{[z_1,z_2)} \Bigl(f(x)+\sum_{i \in \Z} v_i \delta(x - x_i)\Bigr)\rd x  \nn \\
    \EQ \lb{mv-f2} \lim_{n_2-n_1 \to +\oo} \frac{1}{x_{n_2}-x_{n_1}} \int_{[x_{n_1},x_{n_2})} \Bigl(f(x)+\sum_{i \in \Z} v_i \delta(x - x_i)\Bigr)\rd x.
    \eea
That is, if one of limits exists, then the other one exists as well and they are equal. Indeed, suppose that there exists $B>0$ such that
    \be \lb{mv-f3}
    \disp \sup_{x \in \R \setminus \Gamma}|f(x)|+\|V\|_{\infty}=\|f(x)\|_{\infty}+\|V\|_{\infty} \le B.
    \ee
Then by \x{mv-f1}, \x{mv-f3} and \x{lmm}, we have
    \beaa
    \EM \Bigl|\int_{[z_1,z_2)}\ce{^{V}_{\Gamma}}f(x) \rd x-\int_{[x_{n_1},x_{n_2})}\ce{^{V}_{\Gamma}}f(x) \rd x\Bigr|\\
    \LE \Bigl|\int_{[x_{n_2},z_2)}\ce{^{V}_{\Gamma}}f(x) \rd x\Bigr|+\Bigl|\int_{[x_{n_1},z_1)}\ce{^{V}_{\Gamma}}f(x) \rd x\Bigr|\\
    \LE 2(M+1)B <+\infty.
    \eeaa
It follows that
    \beaa
    \EM \lim_{z_2-z_1 \to +\oo} \frac{1}{z_2-z_1} \int_{[z_1,z_2)}\Bigl(f(x)+\sum_{i \in \Z} v_i \delta(x - x_i)\Bigr)\rd x \\
    \EQ \lim_{z_2-z_1 \to +\oo} \frac{x_{n_2}-x_{n_1}}{z_2-z_1} \frac{\disp\int_{[z_1,z_2)}\ce{^{V}_{\Gamma}}f(x) \rd x-\int_{[x_{n_1},x_{n_2})}\ce{^{V}_{\Gamma}}f(x) \rd x+ \int_{[x_{n_1},x_{n_2})}\ce{^{V}_{\Gamma}}f(x) \rd x}{x_{n_2}-x_{n_1}} \\
    \EQ \lim_{n_2-n_1 \to +\oo} \frac{1}{x_{n_2}-x_{n_1}} \int_{[x_{n_1},x_{n_2})}\Bigl(f(x)+\sum_{i \in \Z} v_i \delta(x - x_i)\Bigr)\rd x,
    \eeaa
provided one of limits exists. The claim \x{mv-f2} is deduced.

\medskip
\fbox{Claim 2}\ : introduce the function $F: \mathrm{H}\Bigl(\ce{^{V}_{\Gamma\,}}f\Bigr) \to \K$ by
    \[
    F\Bigl(\ce{^{\ti V}_{\ti \Gamma\,}}\ti f\Bigr):=\int_{[0,\ti x_1)}\Bigl(\ti f(x)+\sum_{i \in \Z} \ti v_i \delta(x - \ti x_i)\Bigr)\rd x, \q \mbox{for~}\ce{^{\ti V}_{\ti\Gamma\,}}\ti f \in \mathrm{H}\Bigl(\ce{^{V}_{\Gamma\,}}f\Bigr).
    \]
We assert that $F$ is continuous. Indeed, assume that
    \be \lb{mv-f4}
    \disp \lim_{k \to +\infty}\dist\Bigl(\ce{^{V_k}_{\Gamma_k}}f_k,\ce{^{V_0}_{\Gamma_0}}f_0\Bigr)= 0,
    \ee
where $\ce{^{V_k}_{\Gamma_k}}f_k \in \mathrm{H}\Bigl(\ce{^{V}_{\Gamma\,}}f\Bigr)$, $k \in \N_0$. By a direct computation, we have
    \be \lb{mv-f45}
    F\Bigl(\ce{^{\ti V}_{\ti \Gamma\,}}\ti f\Bigr)=\int_{[0,\ti x_1)}\ti f(x) \rd x +\ti v_0, \q \mbox{for~}\ce{^{\ti V}_{\ti\Gamma\,}}\ti f \in \mathrm{H}\Bigl(\ce{^{V}_{\Gamma\,}}f\Bigr).
    \ee
Denote $\Gamma_k=\{x^k_i\}_{i \in \Z} \in L_{m,M}(\Z)$ and $V_k=\{v^k_i\}_{i \in \Z} \in \ell^{\infty}(\Z)$. For any $\e>0$, denote
    \be \lb{mv-f5}
    \epsilon:= \frac{\e}{4B+M+1},
    \ee
where $B$ and $M$ are introduced in \x{unbd} and \x{lmm}, respectively. It follows from \x{mv-f4}, Lemma \ref{tr-con}, Remark \ref{re-f}, Lemma \ref{con-l} and \x{dist-v} that there exists $k_\epsilon \in \N$ such that when $k \ge k_\epsilon$, we have
    \be \lb{mv-f6}
    |v^k_0-v^0_0|< \epsilon,\  |x^k_1-x^0_1|< \epsilon,
    \ee
and
    \be \lb{mv-f7}
    |f_k(x)-f_0(x)|< \epsilon, \q \mbox{for~} x \in \R \setminus F_{\epsilon}(\Gamma_0).
    \ee
Then we obtain
    \beaa
    \EM \Bigl |F\Bigl(\ce{^{V_k}_{\Gamma_k}}f_k\Bigr) - F\Bigl(\ce{^{V_0}_{\Gamma_0}}f_0\Bigr)\Bigr|\\
    \LE  \int_{[0,\epsilon]}|f_k(x)-f_0(x)| \rd x + \int_{(\epsilon,x^0_1-\epsilon)}|f_k(x)-f_0(x)| \rd x\\\
    \EM + \Bigl|\int_{[x^0_1-\epsilon,x^k_1)}f_k(x)\rd x- \int_{[x^0_1-\epsilon,x^0_1)}f_0(y)\rd y\Bigr|+ |v^k_0-v^0_0| \\
    \AND{<} 2B\epsilon+M\epsilon+2B\epsilon+\epsilon=\e,
    \eeaa
where \x{mv-f45}, \x{unbd}, \x{mv-f7}, \x{mv-f6} and \x{mv-f5} are used. The claim is deduced.

\medskip
\fbox{Claim 3}\ : the following relation holds:
    \bea
    \EM \lim_{n_2-n_1 \to +\oo} \frac{1}{x_{n_2}-x_{n_1}} \int_{[x_{n_1},x_{n_2})} \Bigl(f(x)+\sum_{i \in \Z} v_i \delta(x - x_i)\Bigr)\rd x  \nn \\
    \EQ \lb{mv-f8} [\Gamma] \lim_{n_2-n_1 \to +\infty} \frac{1}{n_2-n_1} \sum_{\tau=n_1}^{n_2-1} F\Bigl(\ce{^{V}_{\Gamma\,}}f \cdot \tau\Bigr),
    \eea
where $[\Gamma]$ is introduced in Lemma \ref{den-l}. That is, if one of the limits exists, then the other one exists as well and they are equal. Indeed, we have
    \beaa
    \EM  \int_{[x_{n_1},x_{n_2})} \Bigl(f(x)+\sum_{i \in \Z} v_i \delta(x - x_i)\Bigr)\rd x\\
    \EQ  \sum_{\tau=n_1}^{n_2-1}\int_{[x_\tau,x_{\tau+1})} \Bigl(f(x)+\sum_{i \in \Z} v_i \delta(x - x_i)\Bigr)\rd x \\
    \EQ  \sum_{\tau=n_1}^{n_2-1}\Bigl(\int_{[x_\tau,x_{\tau+1})} f(x)\rd x + v_\tau\Bigr) \\
    \EQ  \sum_{\tau=n_1}^{n_2-1}F\Bigl(\ce{^{V}_{\Gamma\,}}f \cdot \tau \Bigr),
    \eeaa
where \x{sh-tr} and \x{mv-f45} are used. The relation \x{mv-f8} is deduced by Lemma \ref{den-l}.

The uniform convergence of the limit \x{mv-f0} is obtained by Lemma \ref{atg} {\romannumeral4}) and Remark~\ref{atg-co}.
\end{proof}

\begin{remark}
{\rm
If there are no $\delta$-interactions, then the mean value of $\prescript{0}{\Gamma}f \in PC_{0,ap}(\R)$ may be defined by
    \[
    \mathrm{M}\Bigl(\ce{^{0}_{\Gamma}}f\Bigr) := \lim_{z_2-z_1 \to +\oo} \frac{1}{z_2-z_1} \int_{[z_1,z_2)} f(x)\rd x.
    \]
Note that the value of the integral is the same if $[z_1,z_2)$ is replaced by $[z_1,z_2]$.
}
\end{remark}

\begin{remark}
{\rm For $\prescript{V}{\Gamma\,}f \in PC_{\delta,0,ap}(\R)$, we have a decomposition formula:
    \be \lb{mv-dec}
    \mathrm{M}\Bigl(\ce{^{V}_{\Gamma\,}}f\Bigr)=\mathrm{M}\Bigl(\ce{^{0}_{\Gamma}}f\Bigr)+[\Gamma]\mathrm{M}(V),
    \ee
where Lemma \ref{lm-pcdap} {\romannumeral3}), Lemma \ref{lap-mv} and Lemma \ref{den-l} are used.}
\end{remark}

\medskip
%
%
%
%
%

\section{Reduction to Skew-Product Dynamical Systems}\label{sec.4}

For autonomous ODEs, the family of solutions with different initial values generates a flow due to the existence and uniqueness of solutions of ODEs. For non-autonomous ODEs, when the hull of the vector field of ODEs is involved, we may construct a skew-product flow provided the existence and uniqueness of solutions of ODEs hold as well; see \cite{Se68, SY98} for the detailed idea. In this section, we will use this idea to construct a skew-product dynamical system from \x{sys}. There are two crucial issues in this setting. One is to overcome the difficulty that is caused by impulses at points of $\Gamma$, and the other one is to show that the dynamical system is continuous on the phase space under the uniform topology.

\subsection{Homotopy and Argument}

We recall some necessary results from \cite{L02, DZ}. Denote by $\mathrm{M}(2,2)$ the space of all $2 \times 2$ real matrices.  Let $J$ be the standard symplectic matrix
    \[
    J:=\left(\begin{array}{cc}
         0 & -1 \\
         1 & 0
       \end{array}\right).
    \]
We say that $D \in \mathrm{M(}2,2)$ is \emph{symplectic} if and only if $D^T J D = J$, where $D^T$ is the transpose matrix of $D$. It is well known that the collection of all $2 \times 2$ real symplectic matrices forms a group with respect to matrix multiplication. Let us denote this group by $\mathrm{Sp}(2,\R)$. It is well known that
    \[
    \mathrm{Sp}(2,\R) = {\rm SL}_2(\R) = \left\{ D \in \mathrm{M}(2,2) : \det(D) = 1 \right\}.
    \]

\begin{lemma} \lb{sp-de} \cite{L02}
For any $D \in \mathrm{Sp}(2,\R)$, there exists a unique decomposition such that $D = A \, U$, where $A \in \mathrm{Sp}(2,\R)$ is a symmetric and positive-definite matrix and $U \in \mathrm{Sp}(2,\R)$ is an orthogonal matrix. Explicitly, we have:
    \be \lb{dem-d}
    D = \left(
        \begin{array}{cc}
          r & z \\
          z & \frac{1+z^2}{r} \\
        \end{array}
      \right) \left(
                     \begin{array}{cc}
                       \cos \vartheta & -\sin \vartheta \\
                       \sin \vartheta & \cos \vartheta \\
                     \end{array}
                   \right),
    \ee
where $(r, \vartheta, z) \in \R^+ \times  \R/(2\pi \Z-\pi) \times \R$ is uniquely determined by $D$.
\end{lemma}

This implies the following result.

\begin{lemma} \lb{sp-r3} \cite{L02}
There exists a one-to-one correspondence from $\mathrm{Sp}(2,\R)$ to $\{ (x,y,z) \in \R^3 \setminus \{z \mbox{-axis}\}\}$ as
    \[
    g : D \mapsto (r \cos \vartheta, r \sin \vartheta, z),
    \]
where $(r, \vartheta, z)$ is defined above. Moreover, $g$ is a homeomorphism.
\end{lemma}

Due to the expression of \x{sys}, we only consider the following group denoted by
    \[
    \mathrm{Trig}(2,\R):=\left\{R_c:=\left(
                               \begin{array}{cc}
                                 1 & c \\
                                 0 & 1 \\
                               \end{array}
                             \right): c \in \R
    \right\} \subset \mathrm{Sp}(2,\R).
    \]
For $R_c \in \mathrm{Trig}(2,\R)$, the unique decomposition can be calculated as
    \[
    R_c = \left(
        \begin{array}{cc}
          \frac{c^2+2}{\sqrt{c^2+4}} & \frac{c}{\sqrt{c^2+4}} \\
          \frac{c}{\sqrt{c^2+4}} & \frac{2}{\sqrt{c^2+4}} \\
        \end{array}
      \right) \left(
                     \begin{array}{cc}
                       \frac{2}{\sqrt{c^2+4}} & \frac{c}{\sqrt{c^2+4}} \\
                       -\frac{c}{\sqrt{c^2+4}} & \frac{2}{\sqrt{c^2+4}} \\
                     \end{array}
                   \right).
    \]
Construct a continuous path $P_c (\cdot) : [0,1] \rightarrow \mathrm{Sp}(2,\R)$ as
    \be \lb{homo}
    P_c(\tau) = \left(
        \begin{array}{cc}
          \frac{(\tau c)^2+2}{\sqrt{(\tau c)^2+4}} & \frac{\tau c}{\sqrt{(\tau c)^2+4}} \\
          \frac{\tau c}{\sqrt{(\tau c)^2+4}} & \frac{2}{\sqrt{(\tau c)^2+4}} \\
        \end{array}
      \right) \left(
                     \begin{array}{cc}
                       \frac{2}{\sqrt{(\tau c)^2+4}} & \frac{\tau c}{\sqrt{(\tau c)^2+4}} \\
                       -\frac{\tau c}{\sqrt{(\tau c)^2+4}} & \frac{2}{\sqrt{(\tau c)^2+4}} \\
                     \end{array}
                   \right)=\left(
                             \begin{array}{cc}
                               1 & \tau c \\
                               0 & 1 \\
                             \end{array}
                           \right).
    \ee
$P_c(\cdot)$ connects $I_2$ and $R_c$. The homotopy class of $P_c(\cdot)$ is denoted by $[P_c]$. Then the jump of arguments on $\Gamma$ can be well defined when the homotopy class is fixed by the construction $[P_c]$; see \cite[Figure 1]{DZ}. In detail, denote by $\mathrm{V}(\R^2)$ the set of all vectors starting from the origin in $\R^2$. The equivalence $\sim$ on $\mathrm{V}(\R^2)$ is defined by
    \[
    \vec{v}_1 \sim \vec{v}_2 \Longleftrightarrow \vec{v}_1=k \vec{v}_2, \quad \mbox{for~some~}k \in \R^+.
    \]
It is well known that
    \[
    \mathrm{L}(\R) := \mathrm{V}(\R^2)/ \sim
    \]
is an orientable compact manifold of dimension one, and may be regarded as a two-covering of the real projective line $\mathbb{RP}^1$. Topologically, $\mathrm{L}(\R)$ is homeomorphic to the circle ${\mathbb S}_{2\pi} := \R / 2 \pi \Z$.

Let $\Xi \in \R$. Then we have
    \[
    P_c(\tau)(\cos \Xi, \sin \Xi)^T=(\cos \Xi+ \tau c\sin \Xi, \sin \Xi)^T.
    \]
Since the homotopy class of $P_c(\cdot)$ is fixed and $\arg(\cdot)$ is understood as a continuous branch, the argument function
    \[
    F(c,\tau,\Xi) = \arg(\cos \Xi+ \tau c \sin \Xi+\mathrm{i} \sin \Xi)
    \]
is continuous with respect to $(c, \tau, \Xi) \in \R \times [0,1] \times \R$. In particular, we may choose one continuous branch of $F(c,\tau,\Xi)$ such that when $\tau = 0$, we have
    \[
    \arg(\cos \Xi+ \mathrm{i} \sin \Xi) = \Xi.
    \]
Then we define the jump of arguments by
    \be  \lb{j-def}
    J(c,\Xi) = F(c,1,\Xi)-F(c,0,\Xi).
    \ee


\begin{lemma} \lb{j-cont} \cite{DZ}
$J : \R^2 \to \R$ is continuous with respect to $(c, \Xi) \in \R^2$. Moreover,
    \[
    J(c, \Xi + 2 \pi) = J(c, \Xi).
    \] 
\end{lemma}

We revisit now the system \x{sys}. Let $\ce{^{V}_{\Gamma}}q \in PC_{\delta,m,M,ap}(\R)$. We need to embed it in a family of systems as follows:
    \begin{equation} \label{sys-fa}
    \left\lbrace \begin{array}{ll}
    \disp\frac{\rd}{\rd x}\left(
                       \begin{array}{c}
                         \psi' \\
                         \psi \\
                       \end{array}
                     \right) =\left(
                                \begin{array}{cc}
                                  0 & \tilde{q}(x)- E \\
                                  1 & 0 \\
                                \end{array}
                              \right)\left(
                                       \begin{array}{c}
                                         \psi' \\
                                         \psi \\
                                       \end{array}
                                     \right) , \qquad &x \in \mathbb{R}\setminus \tilde{\Gamma},\\
    \left(
      \begin{array}{c}
        \psi'(\tilde{x}_n+) \\
        \psi(\tilde{x}_n+) \\
      \end{array}
    \right)= \left(
               \begin{array}{cc}
                 1 & \tilde{v}_n \\
                 0 & 1 \\
               \end{array}
             \right)\left(
                      \begin{array}{c}
                         \psi'(\tilde{x}_n-) \\
                        \psi(\tilde{x}_n-) \\
                      \end{array}
                    \right)
    , \qquad &x=\tilde{x}_n \in \tilde{\Gamma},
    \end{array} \right.
    \end{equation}
where $\ce{^{\ti V}_{\ti \Gamma}}\ti q \in \mathrm{H}\Bigl(\ce{^{V}_{\Gamma\,}}q \Bigr)$. For definiteness, the solution of \x{sys-fa} is understood to be right-continuous with respect to $x \in \R$, that is, $(\psi'(x+), \psi(x+))^T \equiv (\psi'(x), \psi(x))^T$. In this sense, $\psi'(x)$ and $\psi(x)$ are well defined on $\R$. Suppose that $\Psi(x) := \Psi_{E}\Bigl(x;\ce{^{\ti V}_{\ti \Gamma\,}}\ti q\Bigr)$ is the fundamental matrix solution of \x{sys-fa} with the initial value $\Psi(0) = I_2$. Then we have the following result.

\begin{lemma} \cite{DZ}
For any $x \in \R$, $\Psi(x) \in \mathrm{Sp}(2,\R)$.
\end{lemma}

If $(\psi'(x), \psi(x))^T$ has the initial value $(\psi'(0),\psi(0))^T = (\alpha,\beta)^T$, we have
    \[
    (\psi'(x),\psi(x))^T = \Psi(x) (\alpha,\beta)^T. 
    \]
Introduce the so-called Pr\"{u}fer transformation as
    \be \label{pufer}
    \psi'+ \mathrm{i}\, \psi= r\ \mathrm{e}^{\mathrm{i}\, \theta}.
    \ee
Then the argument $\theta = \theta(x)$ may be denoted by
    \[
    \theta(x) := \arg (\psi'(x) + \mathrm{i}\, \psi(x)),
    \]
where $(\psi'(x), \psi(x))^T$ is any non-trivial solution of \x{sys-fa}. When the system \x{sys-fa} is restricted on $\mathbb{R} \setminus \tilde \Gamma$, we understand $\arg(\cdot)$ as a continuous branch on $[\tilde x_n, \tilde x_{n+1})$, where $\tilde \Gamma= \{\tilde x_n\}_{n \in \Z}$. It is easy to obtain that the differential equation for $\theta$ is
    \[
    \theta'(x) = \cos^2 \theta - (\tilde q(x) - E) \sin^2 \theta, \qq x \in \mathbb{R} \setminus \tilde \Gamma.
    \]
But it is crucial to deal with the jump of arguments on $\Gamma$ via a reasonable approach, because the vector field of \x{sys-fa} on $\tilde \Gamma$ is singular. To overcome this difficulty, we use the homology that is defined by \x{homo}. Thus via the Pr\"{u}fer transformation \x{pufer}, the evolution of the arguments is found to be
    \begin{equation} \label{ar-fa}
    \left\lbrace \begin{array}{ll}
    \theta'(x) = \cos^2 \theta(x) - (\tilde{q}(x)-E) \sin^2 \theta(x), \qq & x \in \mathbb{R} \setminus \tilde{\Gamma}, \\
    \theta(\tilde{x}_n+) - \theta(\tilde{x}_n-) = J(\tilde{v}_n,\theta(\tilde{x}_n-)), \qq & x = \tilde{x}_n \in \tilde{\Gamma}.
    \end{array} \right.
    \end{equation}
Denote by $\theta_E(x) = \theta_{E}\Bigl(x+;\ce{^{\ti V}_{\ti \Gamma\,}}\ti q,\Xi\Bigr)$ the solution of \x{ar-fa} with the initial value $\theta_E(0) = \Xi \in \R$. By the uniqueness of solutions of ODEs and the boundedness of piecewise continuous functions, it is easy to deduce the following result.

\begin{lemma} \lb{lm-ar-mix}
Let $\ce{^{\ti V}_{\ti \Gamma\,}}\ti q \in \mathrm{H}\Bigl(\ce{^{V}_{\Gamma\,}}q \Bigr)$ and $E \in \R$ be fixed. Then
\begin{itemize}
\item for $\Xi \in \R$, $x \in \R$ and $k \in \Z$, we have
    \be \lb{th-pe}
    \theta_{E} \Bigl(x; \ce{^{\ti V}_{\ti \Gamma\,}}\ti q, \Xi + 2 k \pi\Bigr) - (\Xi + 2 k \pi) = \theta_{E} \Bigl(x; \ce{^{\ti V}_{\ti \Gamma\,}}\ti q, \Xi\Bigr) - \Xi;
    \ee
\item for $x \in \R$, we have
    \be \lb{th-con}
    \theta_{E} \Bigl(x; \ce{^{\ti V}_{\ti \Gamma}}\ti q, \Xi_1\Bigr) < \theta_{E} \Bigl(x; \ce{^{\ti V}_{\ti \Gamma}}\ti q, \Xi_2\Bigr), \qq \mbox{when~} \Xi_1< \Xi_2,
    \ee
\item for $\Xi \in \R$, $k_1, k_2 \in \Z$ and $\tilde{\Gamma} = \{ \tilde{x}_n \}_{n \in \Z}$, we have
    \be \lb{th-sk}
    \theta_{E} \Bigl(\ti{x}_{k_1 + k_2}; \ce{^{\ti V}_{\ti \Gamma\,}}\ti q, \Xi\Bigr) = \theta_{E} \Bigl(\ti x_{k_1 + k_2} - \ti x_{k_2}; \ce{^{\ti V}_{\ti \Gamma\,}}\ti q \cdot k_2, \theta_{E} \Bigl(\ti{x}_{k_2}; \ce{^{\ti V}_{\ti \Gamma\,}}\ti q, \Xi\Bigr)\Bigr),
    \ee
where $\ce{^{\ti V}_{\ti \Gamma\,}}\ti q \cdot k_2$ is defined by \xx{sh-tr};

\item for $\tilde{\Gamma} = \{ \tilde{x}_n \}_{n \in \Z}$, we have
    \be \lb{th-di}
    \lim_{x\to+\oo} \frac{\theta_{E} \Bigl(x; \ce{^{\ti V}_{\ti \Gamma\,}}\ti q, \Xi\Bigr) - \Xi}{x} = \lim_{n \to +\oo} \frac{\theta_{E} \Bigl(\tilde{x}_n; \ce{^{\ti V}_{\ti \Gamma\,}}\ti q, \Xi\Bigr) - \Xi}{\tilde{x}_n},
    \ee
that is, if one of limits exists, then the other one exists as well and they are equal.
\end{itemize}
\end{lemma}

The proof of Lemma \ref{lm-ar-mix} is omitted. For details, see \cite{DZ}. However since the uniform topology is weaker than the one considered in \cite{DZ}, we will give the proof of the following result.

\begin{lemma} \lb{lm-ar-con}
For $\tilde{\Gamma} = \{ \tilde{x}_n \}_{n \in \Z}$, let $k \in \Z$ be fixed. Then $\theta_{E} \Bigl(\ti{x}_k; \ce{^{\ti V}_{\ti \Gamma}}\ti q , \Xi\Bigr) : \mathrm{H}\Bigl(\ce{^{V}_{\Gamma}}q \Bigr) \times \R \to \R$ is continuous.
\end{lemma}

\begin{proof}
Without loss of generality, we only check the case $k = 1$. For the general case, we may obtain the result by induction. Assume that
    \be \lb{lm-ar-con0}
    \disp \lim_{i\to +\infty}\dist\Bigl(\ce{^{\ti V_i}_{\ti \Gamma_i}}\ti q_i,\ce{^{\ti V_0}_{\ti \Gamma_0}}\ti q_0\Bigr)= 0 \quad \mbox{and} \quad \lim_{i\to +\infty}|\Xi_i-\Xi_0|=0,
    \ee
where $\Bigl(\ce{^{\ti V_i}_{\ti \Gamma_i}}\ti q_i , \Xi_i \Bigr) \in \mathrm{H}\Bigl(\ce{^{V}_{\Gamma}}q \Bigr) \times \R, \ i \in \N_0$. Then for any $\e>0$, there exists $i_{\e} \in \N$ such that Remark \ref{re-f} holds and meanwhile $|\Xi_i-\Xi_0| < \e$ when $i \ge i_\e$. Let
    \[
    \theta_i (x) := \theta_{E}\Bigl(x; \ce{^{\ti V_i}_{\ti \Gamma_i}}\ti q_i , \Xi_i \Bigr),
    \]
where $\tilde{\Gamma}_i = \{ \ti x^i_n \}_{n \in \Z}$. Then we have
    \be  \lb{lm-ar-con05} \left \{ \begin{array}{ll}
    \theta_i'(x) = \cos^2 \theta_i(x) - (\tilde{q}_i(x) - E) \sin^2 \theta_i(x), \qq x \in (0, \ti x^i_1), \\
    \theta_i(0) = \Xi_i. \end{array} \right.
    \ee
For $\e >0$ given in Remark \ref{re-f} and $x \in [\e,\ti x^0_1-\e]$, this implies that
    \be
    \theta_i(x)=\Xi_i + \left( \int_{(0,\e]} + \int_{[\e,x]}\right) \cos^2 \theta_i(\tau) - (\tilde{q}_i(\tau) - E) \sin^2 \theta_i(\tau) \rd \tau. \nn
    \ee
Denote $D(x):=\theta_i(x)-\theta_0(x), \ x \in [\e,\ti x^0_1-\e]$. We have
    \be
    D(x) = (\Xi_i - \Xi_0) + D_1 + D_2(x). \nn
    \ee
where
    \be
    D_1:=\int_{(0,\e]} \bigl(\cos^2 \theta_i(\tau) - (\tilde{q}_i(\tau) - E) \sin^2 \theta_i(\tau) \bigr)
    - \bigl( \cos^2 \theta_0(\tau) - (\tilde{q}_0(\tau) - E) \sin^2 \theta_0(\tau)\bigr) \rd \tau \nn
    \ee
and
    \be
    D_2(x) := \int_{[\e,x]} \bigl(\cos^2 \theta_i(\tau) - (\tilde{q}_i(\tau) - E) \sin^2 \theta_i(\tau)\bigr) - \bigl(\cos^2 \theta_0(\tau) - (\tilde{q}_0(\tau) - E) \sin^2 \theta_0(\tau)\bigr) \rd \tau. \nn
    \ee

By the uniform boundedness of $(\tilde{q}_i, \ti \Gamma_i)$, we know that there exists $C_1 > 0$ such that
    \be \lb{lm-ar-con5}
    |D_1| < C_1 \e.
    \ee

Now we consider the term $D_2(x)$. When $x$ is fixed, we can regard $\theta_0(x)$ and $\theta_i(x)$ as two real numbers. By the mean value theorem, there exist $\zeta(x), \ \eta(x)$ which belong to the interval with endpoints $\theta_0(x)$ and $\theta_i(x)$ such that
    \bea
    \cos \theta_i (x) - \cos \theta_0(x) \EQ -\sin \zeta(x)(\theta_i(x)-\theta_0(x)), \nn \\
    \sin \theta_i (x) - \sin \theta_0(x) \EQ \cos \eta(x) (\theta_i(x)-\theta_0(x)). \nn
    \eea
Then we have
    \be
    D(x)=(\Xi_i-\Xi_0)+D_1+ \int_{[\e,x]}A(\tau) D(\tau)+B(\tau) \rd \tau, \nn
    \ee
where
    \be \lb{lm-ar-con6}
    A(\tau):= -\sin \zeta(\tau) (\cos \theta_i(\tau)+\cos \theta_0(\tau))-(\ti{q}_0(\tau)-E)\cos \eta(\tau) (\sin \theta_i(\tau) + \sin \theta_0(\tau)),
    \ee
and
    \be \lb{lm-ar-con7}
    B(\tau):=(\ti{q}_0(\tau)-\ti{q}_i(\tau)) \sin^2 \theta_i(\tau).
    \ee
It follows that
    \be
    |D(x)|\le |\Xi_2-\Xi_1| + |D_1|+ \int_{[\e,x]}|B(\tau)| \rd \tau +\int_{[\e,x]}|A(\tau)| |D(\tau)| \rd \tau, \quad \e \le x \le  \ti x^0_1-\e. \nn
    \ee
Denote $C(x):=|\Xi_2-\Xi_1| + |D_1|+ \int_{[\e,x]}|B(\tau)| \rd \tau$. Note that $\ti x^0_1 = \Delta \ti x^0_1 \le M$. Then there exists $C_2 > 0$ such that
    \be \lb{lm-ar-con8}
    C(x) < |\Xi_2-\Xi_1| + |D_1|+ M \e < C_2 \e, , \quad \mbox{when~} i \ge i_\e,
    \ee
where \x{lm-ar-con0}, \x{lm-ar-con5}, \x{lm-ar-con7} and Lemma \ref{tr-con} are used.
By the generalized Gronwall inequality \cite[Lemma 6.2 in Chapter I]{Ha80}, we obtain
    \be \lb{th-pe4}
    |D(x)|\le C(x) +\int_{[\e,x]}|A(\tau)| |C(\tau)|\left(\exp \int_{[\tau,x]}|A(u)| \rd u\right) \rd \tau, \quad \e \le x \le \ti x^0_1-\e. \nn
    \ee
Taking $x=\ti x^0_1-\e$, we have
    \be
    |D(\ti x^0_1-\e)|\le  C_2 \e\left(1 +\int_{[0,M]}2(1+|\ti{q}_0(\tau)-E|) \left(\exp \int_{[0,M]}2(1+|\ti{q}_0(u)-E|)\rd u\right) \rd \tau\right), \nn
    \ee
where \x{lm-ar-con6} and \x{lm-ar-con8} are used. Since $\ti{q}_0$ is bounded, there exists $C_3 > 0$ such that
    \[
    |\theta_i(\ti x^0_1-\e)-\theta_0(\ti x^0_1-\e)|=|D(\ti x^0_1-\e)| < C_3 \e , \quad \mbox{when~} i \ge i_\e.
    \]
Again by \x{lm-ar-con05}, we have
    \bea
    \theta_i(\ti x^i_1-)\EQ \theta_i(\ti x^0_1-\e) + \int_{[\ti x^0_1-\e,\ti x^i_1)}\cos^2 \theta_i(\tau) - (\tilde{q}_i(\tau) - E) \sin^2 \theta_i(\tau) \rd \tau, \nn \\
    \theta_0(\ti x^0_1-)\EQ \theta_0(\ti x^0_1-\e) + \int_{[\ti x^0_1-\e,\ti x^0_1)}\cos^2 \theta_0(\tau) - (\tilde{q}_0(\tau) - E) \sin^2 \theta_0(\tau) \rd \tau. \nn
    \eea
By Lemma \ref{tr-con}, we know that $\ti x^i_1 \in [\ti x^0_1-\e,\ti x^0_1+\e]$. Then there exists $C_4 > 0$ such that
    \[
    |\theta_i(\ti x^i_1-)-\theta_0(\ti x^0_1-)|< C_4 \e, , \quad \mbox{when~} i \ge i_\e.
    \]
Furthermore, we have
    \[
    \theta_i(\tilde{x}^i_1) = \theta_i(\tilde{x}^i_1-) + J(\tilde{v}^i_{1}, \theta_E(\tilde{x}^i_1-)).
    \]
By Lemma \ref{j-cont}, we obtain the continuity of $\theta_{E} \Bigl(\ti{x}_1; \ce{^{\ti V}_{\ti \Gamma}}\ti q , \Xi\Bigr)$ with respect to $\ce{^{\ti V}_{\ti \Gamma}}\ti q$ and $\Xi$, finishing the proof.
\end{proof}

\subsection{Skew-Products}

Following the idea in \cite{Se68, SY98, DZ}, we may construct a skew-product dynamical system from \x{ar-fa}. The difference from \cite{DZ} is that we must show the continuity of skew-products under the uniform topology.

Let ${\mathbb S}_{2\pi} := \R / 2 \pi \Z$ and $\mathrm{Z} := \mathrm{H}\Bigl(\ce{^{V}_{\Gamma}}q \Bigr) \times {\mathbb S}_{2\pi}$. We introduce a distance on the product space $\mathrm{Z}$ as
    \begin{eqnarray} \lb{di-prod}
    \dist\Bigl(\Bigl(\ce{^{\ti V_1}_{\ti \Gamma_1}}\ti q_1 , \vartheta_1\Bigr), \Bigl(\ce{^{\ti V_2}_{\ti \Gamma_2}}\ti q_2, \vartheta_2\Bigr)\Bigr)
    :=  \max \Bigl\{ \dist \Bigl(\ce{^{\ti V_1}_{\ti \Gamma_1}}\ti q_1, \ce{^{\ti V_2}_{\ti \Gamma_2}}\ti q_2 \Bigr), |\vartheta_1 - \vartheta_2|_{{\mathbb S}_{2\pi}}\Bigr\}
    \end{eqnarray}
where $\Bigl(\ce{^{\ti V_i}_{\ti \Gamma_i}}\ti q_i, \vartheta_i \Bigr) \in \mathrm{Z},\ i = 1, 2$, and $\dist$ in the right-hand side is constructed by \x{un-dist}. It follows from Definition \ref{ap} {\romannumeral2}) that $(\mathrm{Z}, \dist)$ is a compact metric space. The family of skew-product transformations $\Bigl\{\Phi_E^k\Bigr\}_{k \in \Z}$ on $\mathrm{Z}$ is constructed by
    \be \lb{df-phi}
    \Phi^k_{E} \Bigl(\ce{^{\ti V}_{\ti \Gamma}}\ti q, \vartheta\Bigr) := \Bigl(\ce{^{\ti V}_{\ti \Gamma}}\ti q \cdot k, \theta_{E}\Bigl(\ti x_k; \ce{^{\ti V}_{\ti \Gamma}}\ti q, \Xi\Bigr) \mod 2 \pi \Bigr),
    \ee
where $k \in \Z$, $\Bigl(\ce{^{\ti V}_{\ti \Gamma}}\ti q, \vartheta\Bigr) \in \mathrm{Z}$, and there exists $\Xi \in \R$ satisfying $\vartheta = \Xi \mod 2 \pi$. By \x{th-pe}, $\Phi^k_{E}$ is well defined for each $k \in \Z$. Moreover, we have

\begin{lemma} \lb{phi-ds}
$\Bigl\{ \Phi^k_{E} \Bigr\}_{k \in \Z}$ is a continuous skew-product dynamical system on $\mathrm{Z}$.
\end{lemma}

\begin{proof}

First we assert that $\Bigl\{ \Phi^k_{E} \Bigr\}$ possesses a group structure. In fact, assume that $\Bigl(\ce{^{\ti V}_{\ti \Gamma}}\ti q, \vartheta\Bigr) \in \mathrm{Z}$ and there exists $\Xi \in \R$ satisfying $\vartheta = \Xi \mod 2 \pi$. By \x{df-phi}, \x{dy-tr} and \x{th-sk}, we have
    \beaa
    \EM \Phi_{E}^{k_1} \circ \Phi_{E}^{k_2} \Bigl(\ce{^{\ti V}_{\ti \Gamma}}\ti q, \vartheta\Bigr) \nn \\
    \EQ \Phi_{E}^{k_1} \Bigl(\ce{^{\ti V}_{\ti \Gamma}}\ti q \cdot k_2, \theta_{E}\Bigl(\ti x_{k_2}; \ce{^{\ti V}_{\ti \Gamma}}\ti q, \Xi\Bigr) \mod 2 \pi\Bigr) \\
    \EQ \Bigl(\ce{^{\ti V}_{\ti \Gamma}}\ti q \cdot k_2 \cdot k_1, \theta_{E} \Bigl(\ti x_{k_1 + k_2} - \ti x_{k_2}; \ce{^{\ti V}_{\ti \Gamma}}\ti q \cdot k_2, \theta_{E} \Bigr(\ti x_{k_2}; \ce{^{\ti V}_{\ti \Gamma}}\ti q, \Xi\Bigr)\Bigr) \mod 2 \pi\Bigr) \\
    \EQ \Phi_{E}^{k_1 + k_2} \Bigl(\ce{^{\ti V}_{\ti \Gamma}}\ti q, \vartheta\Bigr).
    \eeaa

Now we aim to prove that for each $k \in \Z$, $\Phi^k_{E}: \mathrm{Z} \to \mathrm{Z}$ is continuous. We make the following claims.

\medskip
\fbox{Claim 1}\ : $\ce{^{\ti V}_{\ti \Gamma}}\ti q \cdot k: \mathrm{H}\Bigl(\ce{^{V}_{\Gamma}}q \Bigr) \to \mathrm{H}\Bigl(\ce{^{V}_{\Gamma}}q \Bigr)$ is continuous. This is obvious from Lemma \ref{tr-tau}.

\medskip
\fbox{Claim 2}\ : $\theta_{E}\Bigl(\ti x_k; \ce{^{\ti V}_{\ti \Gamma}}\ti q, \Xi\Bigr) \mod 2 \pi :\mathrm{Z} \to {\mathbb S}_{2\pi}$ is continuous. This is obvious from Lemma \ref{lm-ar-con}.

Due to Claim 1 and Claim 2, we obtain the desired result.
\end{proof}

Introduce the observable $F_{E}$ from $\mathrm{Z}$ to $\R$ as
    \be \lb{df-f}
    F_{E} \Bigl(\ce{^{\ti V}_{\ti \Gamma}}\ti q, \vartheta \Bigr) := \theta_{E}\Bigl(\ti{x}_1; \ce{^{\ti V}_{\ti \Gamma}}\ti q, \Xi\Bigr) - \Xi, \quad \Bigl(\ce{^{\ti V}_{\ti \Gamma}}\ti q, \vartheta\Bigr) \in \mathrm{Z},
    \ee
where $\Xi \in \R$ satisfies $\vartheta = \Xi \mod 2 \pi$. By \x{th-pe}, $F_{E} \Bigl(\ce{^{\ti V}_{\ti \Gamma}}\ti q, \vartheta\Bigr)$ is well defined on $\mathrm{Z}$. Furthermore, we have

\begin{lemma}  \lb{f-con}
$F_{E} \Bigl(\ce{^{\ti V}_{\ti \Gamma}}\ti q, \vartheta\Bigr)$ is continuous on $\mathrm{Z}$.
\end{lemma}

\begin{proof}
This is obvious by Lemma \ref{lm-ar-con}.
\end{proof}

By \x{j-def}, we have
    \[
    F_{E} \Bigl(\ce{^{\ti V}_{\ti \Gamma}}\ti q, \vartheta\Bigr) = \theta_{E} \Bigl(\ti{x}_1-; \ce{^{\ti V}_{\ti \Gamma}}\ti q, \Xi\Bigr) - \Xi + J(\ti{c}_1, \theta_E(\ti{x}_1-)).
    \]
where $\Xi \in \R$ satisfies $\vartheta = \Xi \mod 2 \pi$. By the construction above and Lemma \ref{den-l}, we reduce the existence of rotation numbers to that of the following ergodic limit with respect to the skew-product dynamical system $\Bigl\{ \Phi_{E}^k \Bigr\}_{k \in \Z}$.

\begin{lemma} \lb{eqs}
Assume that $\Bigl(\ce{^{\ti V}_{\ti \Gamma}}\ti q, \vartheta\Bigr) \in \mathrm{H}\Bigl(\ce{^{V}_{\Gamma}}q \Bigr) \times {\mathbb S}_{2\pi}$ and $\Xi \in \R$ satisfies $\vartheta = \Xi \mod 2 \pi$. Then we have the following relation:
    \[
    \lim_{n \to +\oo} \frac{\theta_{E} \Bigl(\ti x_n; \ce{^{\ti V}_{\ti \Gamma}}\ti q, \Xi\Bigr) - \Xi}{\ti x_n} = [\Gamma] \lim_{n \to +\oo} \frac{1}{n} \sum_{k=0}^{n-1} F_{E} \Bigl(\Phi_{E}^k \Bigl(\ce{^{\ti V}_{\ti \Gamma}}\ti q, \vartheta\Bigr)\Bigr).
    \]
That is, if one of the limits exists, then the other one exists as well and they are equal.
\end{lemma}

\begin{proof}
By Lemma \ref{den-l}, we have
    \[
    \lim_{n \to +\oo} \frac{\theta_{E} \Bigl(\tilde{x}_n; \ce{^{\ti V}_{\ti \Gamma}}\ti q, \Xi\Bigr) - \Xi}{\tilde{x}_n} = [\Gamma] \lim_{n \to +\oo} \frac{\theta_{E} \Bigl(\tilde{x}_n; \ce{^{\ti V}_{\ti \Gamma}}\ti q, \Xi\Bigr) - \Xi}{n},
    \]
provided one of limits exists. Furthermore,
    \beaa
    \EM \theta_{E} \Bigl(\tilde{x}_n; \ce{^{\ti V}_{\ti \Gamma}}\ti q, \Xi\Bigr) - \Xi \nn \\
    \EQ \sum_{k=0}^{n-1} \Bigl(\theta_{E} \Bigl(\tilde{x}_{k+1}; \ce{^{\ti V}_{\ti \Gamma}}\ti q, \Xi\Bigr) - \theta_{E} \Bigl(\tilde{x}_{k}; \ce{^{\ti V}_{\ti \Gamma}}\ti q, \Xi\Bigr)\Bigr) \\
    \EQ \sum_{k=0}^{n-1} \Bigl(\theta_{E} \Bigl(\tilde{x}_{k+1} - \tilde{x}_{k}; \ce{^{\ti V}_{\ti \Gamma}}\ti q \cdot k, \theta_{E} \Bigl(\tilde{x}_{k}; \ce{^{\ti V}_{\ti \Gamma}}\ti q, \Xi\Bigr)\Bigr) - \theta_{E} \Bigl(\tilde{x}_{k};\ce{^{\ti V}_{\ti \Gamma}}\ti q, \Xi\Bigr)\Bigr) \\
    \EQ \sum_{k=0}^{n-1} F_{E} \Bigl(\ce{^{\ti V}_{\ti \Gamma}}\ti q \cdot k, \theta_{E} \Bigl(\tilde{x}_{k}; \ce{^{\ti V}_{\ti \Gamma}}\ti q, \Xi\Bigr) \mod 2 \pi\Bigr) \\
    \EQ \sum_{k=0}^{n-1} F_{E} \Bigl(\Phi^k_{E} \Bigl(\ce{^{\ti V}_{\ti \Gamma}}\ti q, \vartheta\Bigr)\Bigr).
    \eeaa
The proof is completed.
\end{proof}

\medskip
%
%
%
%
%

\section{Rotation Number}\label{sec.5}

In this section we discuss the existence of the rotation number and its continuous dependence on the spectral parameter $E$. Much of the key preparatory work has already been done. In particular, Lemmas~\ref{phi-ds}--\ref{eqs} will be crucial in the discussion that follows.

We revisit again \x{sys}. We know about the solution that $\psi \in C(\R)$ and $\psi'(x)=\psi'(x+) \in PC(\R)$. Define the right derivative of $\psi$ by
    \[
    D^+\psi(x):=\lim_{h \to 0+} \frac{\psi(x+h)-\psi(x)}{x}.
    \]
$D^+u$ in \cite{Se12} is the same as $u^\bullet$ in \cite{MZ13}.  It follows that $D^+\psi(x)=\psi'(x+)$ and $x \mapsto D^+\psi(x)$ is right-continuous. By the choice of a suitable homotopy defined in \x{homo}, we have a well defined argument as
    \[
    \theta_E(x):=\arg(D^+\psi(x)+\mathrm{i}\psi(x)).
    \]
The evolution of $\theta_E(x)$ is found to be
    \begin{equation}
    \left\lbrace \begin{array}{ll}
    \theta'(x) = \cos^2 \theta(x) - (q(x)-E) \sin^2 \theta(x), \qq & x \in \mathbb{R} \setminus \Gamma, \\
    \theta(x_n) - \theta(x_n-) = J(v_n,\theta(x_n-)), \qq & x = x_n \in \Gamma, \nn
    \end{array} \right.
    \end{equation}
where $J(v,\cdot)$ is defined in \x{j-def}. If the ergodic limit
    \[
    \lim_{x \to +\oo}\frac{\theta_E(x)-\theta_E(0)}{x}
    \]
exists, then we call it the \emph{rotation number} of \x{sys} and denote it by $\rho(E)$.

Our goal is to show that the rotation number indeed exists and that it depends continuously on $E$. We devote a subsection to each of these two items. After we have overcome the difficulty caused by jump discontinuities and $\delta$-interactions, the remainder of the proof follows the spirit of Johnson and Moser. For the sake of readability, we present the following fundamental ergodic theorem due to Johnson and Moser.

\begin{lemma} \lb{JM-thm} {\rm  \cite{JM82}}
Let $\{ \varphi^k \}_{k \in \Z}$ be a continuous discrete-time dynamical system on a compact metric space $X$. Then for any continuous function $f$ on $X$ satisfying
    \be \lb{c1}
    \int_X f \, \dmu = 0 \nn
    \ee
for all invariant Borel probability measures $\mu$ under $\{ \varphi^k \}$, one has
    \be \lb{c2}
    \lim_{n \to +\oo} \frac{1}{n} \sum_{k=0}^{n-1} f(\varphi^k(x)) = 0 \nn
    \ee
uniformly for all $x \in X$.
\end{lemma} 

\subsection{Existence}

To show the existence of rotation numbers, inspired by \x{th-di} and Lemma \ref{eqs}, we introduce the following notation.
    \begin{equation} \label{f-s}
    F_{E}^* \Bigl(\ce{^{\ti V}_{\ti \Gamma}}\ti q, \vartheta\Bigr) := \lim_{n \to +\oo} \frac{1}{n} \sum_{k=0}^{n-1} F_{E} \Bigl(\Phi_{E}^k \Bigl(\ce{^{\ti V}_{\ti \Gamma}}\ti q, \vartheta\Bigr)\Bigr), \quad \Bigl(\ce{^{\ti V}_{\ti \Gamma}}\ti q, \vartheta\Bigr) \in \mathrm{Z}, \nn
    \end{equation}
provided the limit exists. For $\ce{^{\ti V}_{\ti \Gamma}}\ti q \in \mathrm{H}\Bigl(\ce{^{V}_{\Gamma}}q \Bigr)$ and $\Xi \in \R$, denote
    \begin{equation} \label{f-dia}
    F_{E}^\diamond \Bigl(\ce{^{\ti V}_{\ti \Gamma}}\ti q, \Xi\Bigr) := \lim_{n \to +\oo} \frac{\theta_{E} \Bigl(\ti x_n; \ce{^{\ti V}_{\ti \Gamma}}\ti q, \Xi\Bigr) - \Xi}{\ti x_n},
    \end{equation}
provided the limit exists. By \x{th-pe}, we obtain the following:

\begin{lemma} \label{inv-th}
If $F_{E}^\diamond \Bigl(\ce{^{\ti V}_{\ti \Gamma}}\ti q, \Xi_0\Bigr)$ exists for $\Xi_0 \in \R$, then $F_{E}^\diamond \Bigl(\ce{^{\ti V}_{\ti \Gamma}}\ti q, \Xi\Bigr)$ exists for all $\Xi \in \R$ and is independent of the choice of $\Xi \in \R$.
\end{lemma}

\begin{proof}
By \x{th-pe}, we know that $F_{E}^\diamond \Bigl(\ce{^{\ti V}_{\ti \Gamma}}\ti q, \Xi_0+2k \pi\Bigr)$ exists for all $k \in \Z$. Then for any $\Xi \in \R$, there exists $k_{\Xi} \in \Z$ such that
    \[
    \Xi_0 + 2 k_{\Xi} \pi \le \Xi < \Xi_0 + 2 (k_{\Xi} + 1) \pi.
    \]
By \x{th-pe} and \x{th-con}, for all $n \in \N$, we have
    \[
    \theta_{E} \Bigl(\ti x_n; \ce{^{\ti V}_{\ti \Gamma}}\ti q, \Xi_0+2k_{\Xi} \pi\Bigr) \le \theta_{E} \Bigl(\ti x_n; \ce{^{\ti V}_{\ti \Gamma}}\ti q, \Xi \Bigr) < \theta_{E} \Bigl(\ti x_n; \ce{^{\ti V}_{\ti \Gamma}}\ti q, \Xi_0+2k_{\Xi} \pi\Bigr) + 2 \pi.
    \]
This implies that for all $\Xi \in \R$, we have
    \[
    F_{E}^\diamond \Bigl(\ce{^{\ti V}_{\ti \Gamma}}\ti q, \Xi \Bigr) \equiv F_{E}^\diamond \Bigl(\ce{^{\ti V}_{\ti \Gamma}}\ti q, \Xi_0 \Bigr) .
    \]
The proof is completed.
\end{proof}

We are now in a position to obtain the main result of this paper.

\begin{theorem} \lb{main}
Let $\ce{^{V}_{\Gamma}}q \in PC_{\delta,m,M,ap}(\R)$. For any $\Xi \in \R$, the limit $F_{E}^\diamond \Bigl(\ce{^{V}_{\Gamma}} q, \Xi \Bigr)$ exists and is independent of the choice of $\Xi \in \R$. Thus the rotation number $\rho(E)$ is well defined.
\end{theorem}

\begin{proof}
By the Krylov-Bogoliubov theorem and Lemma \ref{phi-ds}, there exists an invariant Borel probability measure under $\Bigl\{ \Phi^k_{E} \Bigr\}_{k \in \Z}$, denoted by $\mu$. Then by the Birkhoff ergodic theorem, there exists a Borel set $\mathrm{Z}_\mu \subset \mathrm{Z}$, which depends on the measure $\mu$, such that $\mu (\mathrm{Z}_\mu) = 1$ and $F_{E}^* \Bigl(\ce{^{\ti V}_{\ti \Gamma}}\ti q, \vartheta\Bigr)$ exists for all $\Bigl(\ce{^{\ti V}_{\ti \Gamma}}\ti q, \vartheta\Bigr)\in \mathrm{Z}_\mu$. Furthermore, $F^*_E$ is integrable and satisfies
    \begin{equation} \label{ex-rn2}
    \int_{\mathrm{Z}} F_{E}^*\rd \mu = \int_{\mathrm{Z}} F_{E}\rd \mu =: \rho(E,\mu).
    \end{equation}

Due to Lemma \ref{inv-th}, $\mathrm{Z}_\mu$ can be written in the form $\mathrm{Z}_\mu = \mathrm{E}_\mu \times \mathbb{S}_{2 \pi}$, where $\mathrm{E}_\mu$ is a Borel set in $\mathrm{H}\Bigl(\ce{^{V}_{\Gamma}}q \Bigr)$. Let $\nu$ be the Haar measure on $\mathrm{H}\Bigl(\ce{^{V}_{\Gamma}}q \Bigr)$. Then we have $\nu(\mathrm{E}_\mu) = 1$. By the unique ergodicity of the Haar measure, there exists a set $\hat{\mathrm{E}}_\mu \subset \mathrm{E}_\mu$ such that $\nu (\hat{\mathrm{E}}_\mu) = \mu(\hat{\mathrm{E}}_\mu \times \mathbb{S}_{2 \pi}) = 1$ and $F_{E}^* \Bigl(\ce{^{\ti V}_{\ti \Gamma}}\ti q, \vartheta\Bigr)$ is a constant function on $\hat{\mathrm{E}}_\mu \times \mathbb{S}_{2 \pi}$. It follows from \x{ex-rn2} that the constant must be $\rho(E,\mu)$.

By \x{f-dia}, we see that $\rho(E,\mu)$ in \x{ex-rn2} is independent of the choice of the measure $\mu$. Set $\hat{F}_{E} := F_{E} - \rho(E)$. By Lemma \ref{f-con}, $\hat{F}_E$ is continuous on $\mathrm{Z}$. By \x{ex-rn2}, $\hat{F}_{E}$ satisfies the requirement of Lemma \ref{JM-thm}. Thus, as $k \nearrow +\oo$,
    \be \lb{un-f}
    \lim_{n \to +\oo} \frac{1}{n} \sum_{k=0}^{n-1} \hat{F}_{E} \Bigl(\Phi_{E}^k \Bigl(\ce{^{\ti V}_{\ti \Gamma}}\ti q, \vartheta\Bigr)\Bigr) = \lim_{n \to +\oo} \frac{1}{n} \sum_{k=0}^{n-1} F_{E} \Bigl(\Phi_{E}^k \Bigl(\ce{^{\ti V}_{\ti \Gamma}}\ti q, \vartheta\Bigr)\Bigr) - \rho(E)= 0
    \ee
uniformly for all $\Bigl(\ce{^{\ti V}_{\ti \Gamma}}\ti q, \vartheta\Bigr) \in \mathrm{Z}$.

At last, taking $\ce{^{\ti V}_{\ti \Gamma}}\ti q = \ce{^{V}_{\Gamma}}q$ in \x{un-f}, then by \x{th-di} and Lemma \ref{eqs}, we obtain the existence of the desired limit \x{f-dia}.
\end{proof}

\subsection{Continuity}

The continuity of rotation numbers with respect to the spectral parameter $E$ is crucial in the proof of the gap labeling theorem. We state this result as follows. But the spectrum will be discussed in a future publication. It is worth mentioning that the gap labelling theorem can be established via the Schwartzman homomorphism for operators in $\ell^2(\Z)$; see \cite{DF23}.

\begin{theorem} \lb{rho-con} Let $\ce{^{V}_{\Gamma}}q \in PC_{\delta,m,M,ap}(\R)$ be fixed. Then $\rho(E)$ is continuous with respect to $E \in \R$.
\end{theorem}

\Proof It suffices to show that for each sequence $E_i \to E_0 \in \R$, we have
    \be \lb{conrho1}
    \rho(E_i) \to \rho(E_0) \qq \mbox{as~} i \to +\oo.
    \ee
For each $i \in \N_0$, consider the equation
    \begin{equation} \label{conrho2}
    \left\lbrace \begin{array}{ll}
    \theta'(x) = \cos^2 \theta(x) - (q(x)-E_i) \sin^2 \theta(x), \qq & x \in \mathbb{R} \setminus \Gamma, \\
    \theta(x_n+) - \theta(x_n-) = J(v_n,\theta(x_n-)), \qq & x = x_n \in \Gamma.
    \end{array} \right.
    \end{equation}

Using the argument of Theorem \ref{main}, for each equation \x{conrho2}, we introduce the skew-product dynamical system $\Bigl\{\Phi_{E_i}^k\Bigr\}_{k \in \Z}$ on $\mathrm{Z}= \mathrm{H}\Bigl(\ce{^{V}_{\Gamma}}q \Bigr) \times {\mathbb S}_{2\pi}$ that is defined by
    \be \lb{conrho3}
    \Phi^k_{E_i} \Bigl(\ce{^{\ti V}_{\ti \Gamma}}\ti q, \vartheta\Bigr) := \Bigl(\ce{^{\ti V}_{\ti \Gamma}}\ti q \cdot k, \theta_{E_i}\Bigl(\ti x_k; \ce{^{\ti V}_{\ti \Gamma}}\ti q, \Xi\Bigr) \mod 2 \pi \Bigr), \quad \Bigl(\ce{^{\ti V}_{\ti \Gamma}}\ti q, \vartheta\Bigr) \in \mathrm{Z},
    \ee
where $\Xi \in \R$ satisfies $= \Xi \mod 2\pi$. Meanwhile, the observable function $F_{E_i}$ from $\mathrm{Z}$ to $\R$ is defined by
    \[
    F_{E_i} \Bigl(\ce{^{\ti V}_{\ti \Gamma}}\ti q, \vartheta \Bigr) := \theta_{E_i}\Bigl(\ti{x}_1; \ce{^{\ti V}_{\ti \Gamma}}\ti q, \Xi\Bigr) - \Xi, \quad \Bigl(\ce{^{\ti V}_{\ti \Gamma}}\ti q, \vartheta\Bigr) \in \mathrm{Z}.
    \]
Then by \x{ex-rn2}, we have
    \be \lb{conrho4}
    \rho(E_i)=\int_{\mathrm{Z}} F_{E_i} \rd \mu_i \qq \mbox{for~all~} i \in \N_0,
    \ee
where $\mu_i$ is an invariant Borel probability measure of $\Bigl\{\Phi_{E_i}^k\Bigr\}_{k \in \Z}$ on $\mathrm{Z}$. Since $\mathrm{Z}$ is a compact metric space, by \cite[Theorem 6.5]{Wa82}, we may assume that there exists a Borel probability measure on $\mathrm{Z}$ denoted by $\mu_{\oo}$ such that $\mu_i\rightharpoonup \mu_{\oo}$ in the weak$^\star$ topology.

We assert that $\mu_{\oo}$ is an invariant Borel probability measure of $\Bigl\{\Phi_{E_0}^k\Bigr\}_{k \in \Z}$ on $\mathrm{Z}$. By \cite[Theorem 6.8]{Wa82}, it suffices to show that for each $f \in C(\mathrm{Z},\R)$ (which denotes the space of all continuous functions from $\mathrm{Z}$ to $\R$), we have
    \be \lb{conrho5}
    \int_\mathrm{Z} f \rd \mu_{\oo}= \int_\mathrm{Z} f \circ \Phi_{E_0}^1 \rd \mu_{\oo}.
    \ee
To this end, for any $f \in C(\mathrm{Z},\R)$, denote $g_i:=f \circ \Phi_{E_i}^1, \ i \in \N_0$. By Lemma \ref{phi-ds}, $g_i \in C(\mathrm{Z},\R)$. It follows from \x{conrho2} and the generalized Gronwall inequality \cite[Lemma 6.2 in Chapter I]{Ha80} that $\theta_{E}\Bigl(\ti x_1; \ce{^{\ti V}_{\ti \Gamma}}\ti q, \Xi\Bigr)$ is Lipschitz continuous with respect to $E$. Combining this with the uniform continuity of $f$, we have
    \be \lb{conrho6}
    \lim_{i\to +\infty}\| g_i-g_0\|_{\infty}=0.
    \ee
Since $\mu_i$ is an invariant Borel probability measure of $\Bigl\{\Phi_{E_i}^k\Bigr\}_{k \in \Z}$, we then have
    \[
    \int_\mathrm{Z} f \rd \mu_i= \int_\mathrm{Z} f \circ \Phi_{E_i}^1 \rd \mu_i =\int_\mathrm{Z} g_i \rd \mu_i.
    \]
Taking $i \nearrow +\oo$, by \x{conrho6}, the weak$^\star$ convergence of $\mu_i$ and \cite[Lemma 3.9]{Zh15}, we obtain the desired result \x{conrho5}.

Then again by \x{conrho2} and the generalized Gronwall inequality \cite[Lemma 6.2 in Chapter I]{Ha80}, we may infer that $F_{E} \Bigl(\ce{^{\ti V}_{\ti \Gamma}}\ti q, \vartheta \Bigr)$ is Lipschitz continuous with respect to $E$ as well. By \cite[Lemma 3.9]{Zh15} and \x{conrho4}, we have
    \[
    \lim_{i \to +\infty}\rho(E_i)=\int_\mathrm{Z} F_{E_0} \rd \mu_{\oo}.
    \]
Since $\mu_{\oo}$ is an invariant Borel probability measure of $\Bigl\{\Phi_{E_0}^k\Bigr\}_{k \in \Z}$, we know from \x{conrho4} that
    \[
    \rho(E_0)=\int_\mathrm{Z} F_{E_0} \rd \mu_{\oo}.
    \]
The desired result \x{conrho1} is proved.
\qed

\begin{remark}
{\rm The traditional way of establishing the continuity of the rotation number as a function of $E$ is via spectral theory. Specifically, one uses oscillation theory to connect the rotation number and the integrated density of states. The latter quantity is continuous in $E$ because of the finite-dimensionality of the solution space. We refer the reader to the recent survey \cite{DF23}, where the traditional way is discussed in detail for operators in $\ell^2(\Z)$, which are technically easier to handle. The proof of Theorem~\ref{rho-con} given above is more direct and the discussion takes place entirely on the dynamics side of this correspondence.}
\end{remark}

\medskip
%
%
%
%
%
\section*{Acknowledgments}
We would like to thank Prof. Jialin Hong for helpful discussions. We also thank the editor and anonymous referees for suggestions.

\end{document}